\documentclass[12pt]{amsart}
  
 \usepackage{amsfonts,graphics,amsmath,amsthm,amsfonts,amscd,amssymb,amsmath,latexsym,multicol, 
 mathrsfs}
\usepackage{epsfig,url}
\usepackage{flafter}
\usepackage{fancyhdr}
\makeatletter

\def\jobis#1{FF\fi
  \def\predicate{#1}%
  \edef\predicate{\expandafter\strip@prefix\meaning\predicate}%
  \edef\job{\jobname}%
  \ifx\job\predicate
}

\makeatother

\if\jobis{proposal}%
\else
\fi

 \usepackage[matrix, arrow]{xy}

\DeclareMathOperator{\Supp}{Supp}

\DeclareMathOperator{\Fix}{Fix}

\DeclareMathOperator{\Mov}{Mov}

\DeclareMathOperator{\codim}{codim}

\DeclareMathOperator{\Spec}{Spec}

\DeclareMathOperator{\lct}{lct}

 \numberwithin{equation}{subsection}
 \numberwithin{footnote}{subsection}

 \newtheorem{cor}[subsection]{Corollary}
 \newtheorem{lem}[subsection]{Lemma}
 \newtheorem{prop}[subsection]{Proposition}
 \newtheorem{thm}[subsection]{Theorem}
 \newtheorem{conj}[subsection]{Conjecture}

{%_%_%_%_% upright style; roman (non-italic) text
    \newtheoremstyle{upright}%
        {8pt plus2pt minus4pt}%
        {8pt plus2pt minus4pt}%
        {\upshape}%
        {}%
        {\bfseries\scshape}%
        {}%
        {1em}%
        {}%
\theoremstyle{upright}

 \newtheorem{defn}[subsection]{Definition}

 \newtheorem{rem}[subsection]{Remark}

}

 \newcommand{\x}{\mathscr}
 \newcommand{\C}{\mathbb C}
 \newcommand{\N}{\mathbb N}

 \newcommand{\Q}{\mathbb Q}
 \newcommand{\R}{\mathbb R}
 \newcommand{\Z}{\mathbb Z}
 \newcommand{\bir}{\dashrightarrow}
 \newcommand{\rddown}[1]{\left\lfloor{#1}\right\rfloor} % round-down

\title{Existence of log canonical flips and a special LMMP}
\author{Caucher Birkar}\thanks{
\textbf{Acknowledgements.} This work was carried out while the author was visiting Institut de Math\'ematiques de Jussieu (Paris VI) 
during the 2010-2011 academic year. The visit was financially and administratively supported by the 
Fondation Sciences Math\'ematiques de Paris. He would like to thank these organisations 
for their incredible support and hospitality, and would like to thank Professor Claire Voisin 
who made the visit possible. He is grateful to the referee for the comments and suggestions 
that helped to improve this paper considerably. 
}
\date{\today}
\begin{document}
\maketitle

\begin{abstract}
Let  $(X/Z,B+A)$ be a $\Q$-factorial dlt pair  
where $B,A\ge 0$ are $\Q$-divisors and $K_X+B+A\sim_\Q 0/Z$. 
We prove that any LMMP$/Z$ on $K_X+B$ 
with scaling of an ample$/Z$ divisor terminates with a good log minimal model 
or a Mori fibre space. We show that a more general statement 
follows from the ACC for lc thresholds. 
An immediate corollary of these results is that log flips exist for log canonical pairs. 
\end{abstract}

%\tableofcontents
%%%%%%%%%%%%%%%%%%%%%%%%
%%%%%%%%%%%%%%%%%%%%%%%%%%%%%%%%%%%%%%%%%%%%%%%%%

\section{\textbf{Introduction}}

We work over $\C$. Extending results from the klt case to the lc case is often 
not as easy as it may sound. For example, it took some hard work to prove the 
cone and contraction theorem for lc pairs as done by 
Ambro [\ref{Ambro}] and Fujino [\ref{Fujino-lc-lmmp}]. Another major 
example is the finite generation of log canonical rings: 
the klt case was proved in [\ref{BCHM}] but extending this to the lc case is essentially  
equivalent to proving the abundance conjecture.

It is well-known that log flips exist for klt pairs 
[\ref{BCHM}]. In this paper we study the existence of log flips for lc 
pairs. Along the way, we came across the following more general statement which is more  
suitable for induction, and it is one of the main results of this paper.

\begin{thm}\label{conj-main-general}
Let $(X/Z,B+A)$ be a lc pair where $B,A\ge 0$ are  $\Q$-divisors, $A$ is $\Q$-Cartier, 
and the given morphism $f\colon X\to Z$ is surjective. 
Assume further that $K_X+B+A\sim_\Q 0/Z$. Then, 

$(1)$ $(X/Z,B)$ has a Mori fibre space or a log minimal model $(Y/Z,B_Y)$,

$(2)$ if $K_Y+B_Y$ is nef$/Z$, then it is semi-ample$/Z$,

$(3)$ if $(X/Z,B)$ is $\Q$-factorial dlt, then any LMMP$/Z$ on $K_X+B$ with scaling 
of an ample$/Z$ divisor terminates.
\end{thm}

In the theorem and throughout the paper Mori fibre spaces and  log minimal models 
are meant as in Definitions \ref{d-model} and \ref{d-mfs} which are \emph{slightly different} from the traditional 
definitions.
The surjectivity of $f$ is obviously not necessary but we have 
put there for practical convenience. 
Sketch of the main ideas of the proof of the theorem is given at the beginning of section 6 and the full 
proof is given in the same section.
A weaker form of the theorem was conjectured by
 Koll\'ar [\ref{Kollar}, Conjecture 4.1] in the context of studying singularities.

As mentioned earlier, an immediate consequence of the above theorem concerns the existence of log 
flips for lc pairs.

\begin{cor}\label{t-lc-flips}
Let $(X/Z,B)$ be a lc pair  and $X\to Y/Z$ an extremal $K_X+B$-negative 
flipping contraction. Then, the $K_X+B$-flip of $X\to Y$ exists.  
\end{cor}

The proof of the corollary is given at the end of section 6. 
In view of Ambro [\ref{Ambro}] and 
Fujino [\ref{Fujino-lc-lmmp}], and the previous result, we can 
run the LMMP on any lc pair. Termination of such an LMMP can be reduced to the $\Q$-factorial 
dlt case as in Remark \ref{rem-lifting-flips}.

Arguments very similar to those of the proof of Theorem \ref{conj-main-general} 
also work to prove a more general statement if we assume the ACC conjecture for lc thresholds. 

\begin{conj}[ACC for lc thresholds]\label{ACC}
Suppose that $\Gamma\subseteq [0,1]$ and $S\subseteq
\mathbb{Q}$ are finite sets of rational numbers, and $d$ is a natural number. Then, the set
$$\{\lct(M,X,B)|\mbox{ $(X,B)$ is lc of dimension $\le d$}\}$$

{\flushleft satisfies} the  ACC where the coefficients of $B$ belong to $\Gamma$, $M$ is a 
$\Q$-Cartier divisor with coefficients in $S$,  and $\lct(M,X,B)$ is the lc threshold of $M$ with respect to $(X,B)$.
\end{conj}

The conjecture is often formulated with $\Gamma, S$ being DCC sets but we only need the finite case.
 The conjecture was proved for smooth varieties by de Fernex, Ein and Musta\c{t}\u{a} [\ref{Ein-Mustata}] who used some of the ideas of Koll\'ar [\ref{kollar-acc}].
Hacon, M$\rm ^c$Kernan, and Xu have announced that they have solved the conjecture.

\begin{thm}\label{t-acc-main-general-modified}
Assume Conjecture \ref{ACC} in dimension $d$. 
Let $(X/Z,B)$ be a lc pair of dimension $d$ where $B$ is a $\Q$-divisor and the given morphism 
$f\colon X\to Z$ is surjective. Assume further that 
$K_X+B\sim_\Q 0$ over some non-empty open subset $U\subseteq Z$, and  if $\eta$ is  
the generic point of a lc centre of $(X/Z,B)$, then $f(\eta)\in U$. Then, 

$(1)$ $(X/Z,B)$ has a log minimal model $(Y/Z,B_Y)$,

$(2)$ $K_Y+B_Y$ is semi-ample$/Z$,

$(3)$ if $(X/Z,B)$ is $\Q$-factorial dlt, then any LMMP$/Z$ on $K_X+B$ with scaling 
of an ample$/Z$ divisor terminates.
\end{thm}

The proof is given in section 7 using arguments quite similar to those of section 6.
In section 5, we  show that under finite generation 
a stronger statement holds:

\begin{thm}\label{t-main-under-fg-intro}
Let $(X/Z,B)$ be a $\Q$-factorial dlt pair where $B$ is a $\Q$-divisor and $f\colon X\to Z$ is surjective. 
Assume further that $\mathcal{R}(X/Z,K_X+B)$ is a finitely generated $\mathcal{O}_Z$-algebra, 
and that $(K_X+B)|_{X_\eta}\sim_\Q 0$ where $X_\eta$ is the generic fibre of $f$. 
Then, any LMMP$/Z$ on $K_X+B$ with scaling of an ample$/Z$ divisor 
terminates with a good log minimal model. 
\end{thm}

This in particular takes care of the klt case of Theorems \ref{conj-main-general} and 
\ref{t-acc-main-general-modified}.

\begin{rem}
Hacon and Xu [\ref{Hacon-Xu-1}] have independently obtained more general forms 
of some of our results using a circle of similar ideas. Compare \ref{conj-main-general}, \ref{t-lc-flips}, \ref{t-main-under-fg-intro}, and \ref{t-acc-main-general-modified} 
with [\ref{Hacon-Xu-1}, 1.6, 1.8, 2.11, 1.1].  
\end{rem}

An important ingredient of the proof of Theorems \ref{conj-main-general} and \ref{t-acc-main-general-modified}
is the following result which naturally appears when one tries 
to reduce the semi-ampleness of a lc divisor $K_X+B$ to semi-ampleness in the klt case (cf. [\ref{Keel-Matsuki-McKernan}]).
It is a consequence of Koll\'ar's injectivity theorem and semi-ampleness on semi-lc pairs.
When $Z$ is projective, it follows from Fujino-Gongyo [\ref{Fujino-Gongyo-slc}]; the 
general case was proved by Hacon and Xu [\ref{Hacon-Xu-2}].

\begin{thm}[{[\ref{Fujino-Gongyo-slc}], [\ref{Hacon-Xu-2}]}]\label{conj-s-ampleness}
Let $(X/Z,B)$ be a $\Q$-factorial dlt pair and $T:=\rddown{B}$ where $B$ is 
a $\Q$-divisor. Suppose that 

$\bullet$ $K_X+B$ is nef$/Z$,

$\bullet$ $(K_X+B)|_S$ is semi-ample$/Z$ for each component $S$ of $T$,

$\bullet$ $K_X+B-\epsilon P$ is semi-ample$/Z$ for some $\Q$-divisor 
$P\ge 0$ with $\Supp P=T$ and for any sufficiently small rational number $\epsilon>0$.

Then, $K_X+B$ is semi-ample$/Z$.
\end{thm}

We will now state some of the other results of this paper that are of independent interest.
In section 3, we prove the following which is similar to a result of Fujino 
[\ref{Fujino-dlt-blowup}].

\begin{thm}\label{t-exceptional-LMMP-intro}
Let $(X/Z,B)$ be a $\Q$-factorial dlt pair with $K_X+B\sim_\R M/Z$ where 
$M\ge 0$ is very exceptional$/Z$. 
Then, any LMMP$/Z$ on $K_X+B$ with scaling of an ample$/Z$ divisor 
terminates with a model $Y$ on which we have $K_Y+B_Y\sim_\R M_Y=0/Z$.
\end{thm}

See Definition \ref{d-v-exceptional} for the notion of very exceptional divisors.

In section 4, we prove a stronger version of [\ref{B-II}, Theorem 1.5] which is 
repeatedly used in the subsequent sections:

\begin{thm}\label{t-mmodel-term-scaling-intro}
 Let $(X/Z,B+C)$ be a lc pair of dimension $d$ such that $K_X+B+C$ is nef$/Z$, 
 $B,C\ge 0$ and $C$ is $\R$-Cartier.
Assume that we are given an LMMP$/Z$ on $K_X+B$ with scaling of $C$ as in 
Definition \ref{d-LMMP-scaling} with $\lambda_i$ the corresponding numbers, 
and $\lambda:=\lim_{i\to \infty}\lambda_i$. 
Then, the LMMP terminates in the following cases:

${\rm (i)}$ $(X/Z,B)$ is $\Q$-factorial dlt, $B\ge H\ge 0$ for some ample$/Z$ 
$\R$-divisor $H$,

${\rm (ii)}$ $(X/Z,B)$ is $\Q$-factorial dlt, 
$C\ge H\ge 0$ for some ample$/Z$ $\R$-divisor $H$, and $\lambda>0$,

${\rm (iii)}$ $(X/Z,B+\lambda C)$  has a log minimal model, and $\lambda\neq \lambda_j$ for any $j$.
\end{thm}

 Finally, we briefly mention some previous works on flips. Mori proved the existence 
of flips for 3-folds with terminal singularities [\ref{Mori-flip}]. Shokurov proved 
it in full generality in dimension 3 [\ref{log-flips}][\ref{log-models}], in dimension 
4 in the klt case [\ref{pl-flips}], and also worked out a significant proportion of 
what we know about flips in every dimension. Hacon-M$\rm ^c$Kernan [\ref{HM}] filled in the 
missing parts of Shokurov's program on pl flips. 
Birkar-Cascini-Hacon-M$\rm ^c$Kernan [\ref{BCHM}] (together with Hacon-M$\rm ^c$Kernan [\ref{HM-2}]) proved the problem in the klt case in all dimensions 
(see also [\ref{BP}]). Fujino [\ref{Fujino-fg}] proved it for lc pairs in 
dimension 4. Alexeev-Hacon-Kawamata [\ref{AHK}] and Birkar [\ref{B}] proved it in 
dimension 5 in the klt case.

\vspace{0.3cm}
%%%%%%%%%%%%%%%%%%%%%%%%%%%%%%%%%%%%%
\section{\textbf{Preliminaries}}

\subsection*{\textbf{Notation and basic definitions}}
We work over $k=\C$. 
A \emph{pair} $(X/Z,B)$ consists of normal quasi-projective varieties $X,Z$ over $k$, 
an $\R$-divisor $B$ on $X$ with
coefficients in $[0,1]$ such that $K_X+B$ is $\mathbb{R}$-Cartier, and a projective 
morphism $X\to Z$. For a prime divisor $D$ on some birational model of $X$ with a
nonempty centre on $X$, $a(D,X,B)$
denotes the log discrepancy.

Let $(X/Z,B)$ be a lc pair. By a {$K_X+B$}-flip$/Z$ we mean the flip of a $K_X+B$-negative extremal flipping contraction$/Z$.
A \emph{sequence of log flips$/Z$ starting with} $(X/Z,B)$ is a sequence $X_i\bir X_{i+1}/Z_i$ in which  
$X_i\to Z_i \leftarrow X_{i+1}$ is a $K_{X_i}+B_i$-flip$/Z$, $B_i$ is the birational transform 
of $B_1$ on $X_1$, and $(X_1/Z,B_1)=(X/Z,B)$.
\emph{Special termination} means termination near $\rddown{B}$.

\begin{defn}[Weak lc and log minimal models]\label{d-model}
A pair $(Y/Z,B_Y)$ is a \emph{log birational model} of $(X/Z,B)$ if we are given a birational map
$\phi\colon X\bir Y/Z$ and $B_Y=B^\sim+E$ where $B^\sim$ is the birational transform of $B$ and 
$E$ is the reduced exceptional divisor of $\phi^{-1}$, that is, $E=\sum E_j$ where $E_j$ are the
exceptional/$X$ prime divisors on $Y$. 
A log birational model $(Y/Z,B_Y)$ is a \emph{weak lc model} of $(X/Z,B)$ if

$\bullet$ $K_Y+B_Y$ is nef/$Z$, and

$\bullet$ for any prime divisor $D$ on $X$ which is exceptional/$Y$, we have
$$
a(D,X,B)\le a(D,Y,B_Y)
$$
  A weak lc model $(Y/Z,B_Y)$ is a \emph{log minimal model} of $(X/Z,B)$ if 

$\bullet$ $(Y/Z,B_Y)$ is $\Q$-factorial dlt,

$\bullet$ the above inequality on log discrepancies is strict.\\

A log minimal model  $(Y/Z,B_Y)$ is \emph{good} if $K_Y+B_Y$ is semi-ample$/Z$.
\end{defn}

\begin{defn}[Mori fibre space]\label{d-mfs}
A log birational model $(Y/Z,B_Y)$ of a pair $(X/Z,B)$ is called a \emph{Mori fibre space} if 
$(Y/Z,B_Y)$ is $\Q$-factorial dlt, there is a $K_Y+B_Y$-negative extremal contraction $Y\to T/Z$ 
with $\dim Y>\dim T$, and 
$$
a(D,X,B)\le a(D,Y,B_Y)
$$
for  any prime divisor $D$ (on birational models of $X$) and strict inequality holds if $D$ is on $X$ and contracted$/Y$.
\end{defn}

\begin{defn}[Log smooth model]\label{d-log-smooth-model}
We need to define various kinds of log smooth models that in many situations 
allow us to replace a lc pair with a log smooth one. 
Let $(X/Z,B)$ be a lc pair, and let $f\colon W\to X$ be a log resolution. 
 Let $B_W\ge 0$ be a boundary on $W$ so that $K_W+B_W=f^*(K_X+B)+E$ 
where $E\ge 0$ is exceptional$/X$ and the support of $E$ contains  
each prime exceptional$/X$ divisor $D$ on $W$ if $a(D,X,B)>0$. 
We call $(W/Z,B_W)$ a \emph{log smooth model} of $(X/Z,B)$.
 However, in practice 
we usually need further assumptions. We list the ones we will need:

\emph{Type (1):} We take $B_W$ to be the birational transform of $B$ 
plus the reduced exceptional divisor of $f$, that is, we assume that 
$a(D,W,B_W)=0$  for each prime 
exceptional$/X$ divisor $D$ on $W$.

\emph{Type (2):} We assume that $a(D,W,B_W)>0$ if $a(D,X,B)>0$, for each prime 
exceptional$/X$ divisor $D$ on $W$.

Note that if $(X/Z,B)$ is klt and $(W/Z,B_W)$ is of type (2), then 
$(W/Z,B_W)$ is also klt.
\end{defn}

\begin{defn}[LMMP with scaling]\label{d-LMMP-scaling}
Let $(X_1/Z,B_1+C_1)$ be a lc pair such that $K_{X_1}+B_1+C_1$ is nef/$Z$, $B_1\ge 0$, and $C_1\ge 0$ is $\R$-Cartier. 
Suppose that either $K_{X_1}+B_1$ is nef/$Z$ or there is an extremal ray $R_1/Z$ such
that $(K_{X_1}+B_1)\cdot R_1<0$ and $(K_{X_1}+B_1+\lambda_1 C_1)\cdot R_1=0$ where
$$
\lambda_1:=\inf \{t\ge 0~|~K_{X_1}+B_1+tC_1~~\mbox{is nef/$Z$}\}
$$
Now, if $K_{X_1}+B_1$ is nef/$Z$ or if $R_1$ defines a Mori fibre structure, we stop. 
Otherwise assume that $R_1$ gives a divisorial 
contraction or a log flip $X_1\bir X_2$. We can now consider $(X_2/Z,B_2+\lambda_1 C_2)$  where $B_2+\lambda_1 C_2$ is 
the birational transform 
of $B_1+\lambda_1 C_1$ and continue. That is, suppose that either $K_{X_2}+B_2$ is nef/$Z$ or 
there is an extremal ray $R_2/Z$ such
that $(K_{X_2}+B_2)\cdot R_2<0$ and $(K_{X_2}+B_2+\lambda_2 C_2)\cdot R_2=0$ where
$$
\lambda_2:=\inf \{t\ge 0~|~K_{X_2}+B_2+tC_2~~\mbox{is nef/$Z$}\}
$$
 By continuing this process, we obtain a sequence of numbers $\lambda_i$ and a 
special kind of LMMP$/Z$ which is called the \emph{LMMP$/Z$ on $K_{X_1}+B_1$ with scaling of $C_1$}. 
Note that by definition $\lambda_i\ge \lambda_{i+1}$ for every $i$, and we usually put 
$\lambda=\lim_{i\to \infty} \lambda_i$.
When we refer to \emph{termination with scaling} we mean termination of such an LMMP.  

When we have a lc pair $(X/Z,B)$, we can always find an ample$/Z$ $\R$-Cartier divisor $C\ge 0$ such that 
$K_X+B+C$ is lc and nef$/Z$,  so we can run the LMMP$/Z$ with scaling assuming that all the 
necessary ingredients exist, e.g. extremal rays, log flips. In particular, when 
$(X/Z,B)$ is klt or $\Q$-factorial dlt the required extremal rays and log flips exist 
by [\ref{BCHM}] and [\ref{B}, Lemma 3.1].
\end{defn}

\begin{defn}
Let $f\colon X\to Z$ be a projective morphism of quasi-projective varieties with $X$ normal, 
and $D$ an $\R$-divisor on $X$.
Define the divisorial sheaf algebra of $D$ as 
$$
\mathcal{R}(X/Z,D)=\bigoplus_{m\ge 0}f_*\x{O}_X(\rddown{mD})
$$
where $m\in\Z$. If $Z$ is affine, we take ${R}(X/Z,D)$ to be the global 
sections of $\mathcal{R}(X/Z,D)$.
\end{defn}

By a \emph{contraction} $f\colon X\to Z$ we mean a projective morphism of 
quasi-projective varieties with $f_*\x{O}_X=\x{O}_Z$.

\vspace{0.3cm}
\subsection*{\textbf{Some basic facts about log minimal models and the LMMP}}

We collect some basic properties of log minimal models 
and the LMMP that will be used in this paper.

\begin{rem}[Weak lc model]\label{rem-on-m-models}
 Let $(Y/Z,B_Y)$ be a weak lc model of a lc pair $(X/Z,B)$ and $\phi\colon X\bir Y$ 
the corresponding birational map. Let $f\colon W\to X$ and $g\colon W\to Y$ 
be a common log resolution of $(X/Z,B)$ and $(Y/Z,B_Y)$.
Let
$$
E:=f^*(K_X+B)-g^*(K_Y+B_Y)=\sum_D (a(D,Y,B_Y)-a(D,X,B))D
$$
 where $D$ runs over the prime divisors on $W$.
We will show that $E$ is effective and exceptional$/Y$. From the definition of 
weak lc models we get $f_*E\ge 0$, and since $E$ is antinef$/X$, applying the negativity 
lemma implies that $E\ge 0$. 
Now assume that $D$ is a component of $E$ which is not exceptional$/Y$.
Then, $D$ is exceptional$/X$ otherwise $a(D,X,B)=a(D,Y,B_Y)$ and $D$ could not be a component of $E$. 
By definition of weak lc models we get $a(D,Y,B_Y)=0$, so by effectiveness of $E$ we have 
$a(D,X,B)=0$ which again shows that 
$D$ cannot be a component of $E$. Therefore, $E$ is exceptional$/Y$. 
\end{rem}

\begin{rem}[Two weak lc models]\label{rem-two-wlc-models}
Let $(Y_1/Z,B_{Y_1})$ and $(Y_2/Z,B_{Y_2})$ be two weak lc models of a lc pair $(X/Z,B)$. 
Let $f\colon W\to X$ and $g_i\colon W\to Y_i$ be a common resolution, 
and put
$$
E_i:=f^*(K_X+B)-g_i^*(K_{Y_i}+B_{Y_i})
$$
Then, by Remark \ref{rem-on-m-models}, $E_i$ is effective and exceptional$/Y_i$. 
Since ${g_2}_*(E_1-E_2)\ge 0$ and $E_1-E_2$ is anti-nef$/Y_2$, by the negativity lemma, 
$E_1-E_2\ge 0$. Similarly, $E_2-E_1\ge 0$. Therefore, 
$$
g_1^*(K_{Y_1}+B_{Y_1})=g_2^*(K_{Y_2}+B_{Y_2})
$$ 
In particular, if $K_{Y_1}+B_{Y_1}$ is ample$/Z$, then  $K_{Y_2}+B_{Y_2}$ is semi-ample$/Z$ 
and the birational map $Y_2\bir Y_1$ is actually a morphism which pulls back 
$K_{Y_1}+B_{Y_1}$ to  $K_{Y_2}+B_{Y_2}$.
\end{rem}

\begin{rem}[Log smooth models]\label{rem-log-smooth-model}
Log smooth models satisfy certain nice properties besides being simple in terms of singularities. 
Let $(W/Z,B_W)$ be a log smooth model of a lc pair $(X/Z,B)$. Let $D$ be a prime divisor 
on $W$. Then, $a(D,X,B)\ge a(D,W,B_W)$ with strict inequality iff $D$ is exceptional$/X$ 
and $a(D,X,B)>0$.

Another basic property of $(W/Z,B_W)$ is that  any 
log minimal model of $(W/Z,B_W)$ is also a log minimal model of $(X/Z,B)$. 
Indeed let $(Y/Z,B_Y)$ be a log minimal model of $(W/Z,B_W)$. Let 
$e\colon V\to W$ and $h\colon V\to Y$ be a common resolution. 
Then, 
$$e^*(K_W+B_W)=h^*(K_Y+B_Y)+G
$$
 where $G\ge 0$ is exceptional$/Y$ by Remark \ref{rem-on-m-models}. 
Thus, using 
$$
K_W+B_W=f^*(K_X+B)+E
$$ 
we get 
$$
e^*f^*(K_X+B)=h^*(K_Y+B_Y)+G-e^*E
$$
 Since $G-e^*E$ is antinef$/X$, 
and $f_*e_*(G-e^*E)\ge 0$, by the negativity lemma, $G-e^*E\ge 0$. 
In particular, this means that $a(D,X,B)\le a(D,Y,B_Y)$ for any 
prime divisor $D$ on $V$. Next we will compare log discrepancies of 
prime divisors on $X$ and $Y$.

If $D$ is a prime divisor on $X$ which is exceptional$/Y$, then 
$$
a(D,X,B)=a(D,W,B_W)<a(D,Y,B_Y)
$$ 
On the other hand, let 
$D$ be a prime divisor on $Y$ which is exceptional$/X$. If $D$ is exceptional$/W$, 
then by definition of log minimal models, $a(D,Y,B_Y)=0$. Assume that $D$ is not exceptional$/W$.
Then, $a(D,W,B_W)=a(D,Y,B_Y)$ which implies that $a(D,X,B)\le a(D,W,B_W)$ because of 
the relation $a(D,X,B)\le a(D,Y,B_Y)$ obtained above. This is possible only if 
$$
a(D,X,B)=a(D,W,B_W)=0
$$ 
hence again $a(D,Y,B_Y)=0$. Therefore, the prime exceptional 
divisors of $Y\bir X$ appear with coefficient one in $B_Y$.
Finally, if $D$ is a prime divisor on $Y$ which is not exceptional$/X$, or  
if $D$ is a prime divisor on $X$ which is not exceptional$/Y$,
then 
$$
a(D,X,B)=a(D,W,B_W)=a(D,Y,B_Y)
$$
So, $B_Y$ is the birational transform of $B$ plus the reduced exceptional 
divisor of $Y\bir X$ hence $(Y/Z,B_Y)$ is a log minimal model of $(X/Z,B)$.
\end{rem}

\begin{rem}[Lifting a sequence of log flips with scaling]\label{rem-lifting-flips}
(1) Assume that we are given an LMMP with scaling as in Definition \ref{d-LMMP-scaling} 
which consists of only a sequence $X_i\bir X_{i+1}/Z_i$ of log flips. 
Let $(X_1'/Z,B_1')$ be a $\Q$-factorial dlt blowup of 
$(X_1/Z,B_1)$ and $C_1'$ the pullback of $C_1$ (this exists by Corollary \ref{c-Q-factorial-blup} 
below).
 Since $K_{X_1}+B_1+\lambda_1C_1\equiv 0/Z_1$, we get
$K_{X_1'}+B_1'+\lambda_1C_1'\equiv 0/Z_1$. 

Run an LMMP$/Z_1$ on $K_{X_1'}+B_1'$ 
with scaling of some ample$/Z_1$ divisor which is automatically also an LMMP$/Z_1$ on
$K_{X_1'}+B_1'$ with scaling of 
$\lambda_1C_1'$. Assume that this LMMP terminates with a log minimal model  $(X_2'/Z_1,B_2')$.
By construction, $({X_2}/Z_1,B_2)$ and $(X_2'/Z_1,B_2')$ are both weak lc models of $({X_1'}/Z_1,B_1')$. 
By Remark \ref{rem-two-wlc-models}, 
$X_2'$ maps to $X_2$ and $K_{X_2'}+B_2'$ is the 
pullback of $K_{X_2}+B_2$. Thus, $(X_2'/Z,B_2')$
is a $\Q$-factorial dlt blowup of $(X_2/Z,B_2)$. Now, 
$K_{X_2'}+B_2'+\lambda_1C_2'\equiv 0/Z_1$ where $C_2'$ is the birational 
transform of $C_1'$ and actually the pullback of $C_2$.
We can continue this process: that is use the fact that 
$K_{X_2}+B_2+\lambda_2C_2\equiv 0/Z_2$ and  
$K_{X_2'}+B_2'+\lambda_2C_2'\equiv 0/Z_2$ and run an LMMP$/Z_2$ on $K_{X_2'}+B_2'$, etc.

We have shown that we can lift the original sequence to an LMMP$/Z$ on $K_{X_1'}+B_1'$ 
with scaling of $C_1'$ assuming that the following statement holds for each $i$:\\

 $(*)$  some LMMP$/Z_i$ on $K_{X_i'}+B_i'$ with scaling of some ample$/Z_i$ divisor 
 terminates where $({X_i'}/Z,B_i')$ is a $\Q$-factorial dlt blowup of 
$(X_i/Z,B_i)$.\\

Note however that each $X_i'\bir X_{i+1}'$ is a sequence of log flips and divisorial 
contractions (not necessarily just one log flip).\\

(2) We will show that $(*)$ holds for $i$ if $(X_i/Z,tB_i)$ is klt for some $t\in [0,1]$ 
(in particular, this holds if $(X_i/Z,B_i)$ is klt or $\Q$-factorial dlt). 
 Pick an ample$/Z_i$ $\R$-divisor 
$H\ge 0$ on $X_i$ so that 
$$
K_{X_i}+B_i+H\sim_\R 0/Z_i
$$
 and $(X_i/Z,B_i+H)$ is lc. Now for a 
  sufficiently small $\epsilon>0$, there 
is a small $\delta>0$ and $\Delta_i$ such that 
$$
(1-\delta)B_i\le \Delta_i\sim_\R B_i+\epsilon H/Z_i
$$ 
and that $(X_i/Z,\Delta_i)$ is klt: if $t=1$, then we can just take 
$\Delta_i= B_i+\epsilon H$; but if $t<1$, then $B_i$ is $\R$-Cartier and 
we take $\delta>0$ small enough so that $\delta B_i+\epsilon H$ is ample$/Z_i$, and 
we put $\Delta_i=(1-\delta)B_i+H'$ for a general $H'\sim_\R \delta B_i+\epsilon H/Z_i$.
Now let $K_{X_i'}+\Delta_i'$ be the pullback of $K_{X_i}+\Delta_i$.
We can assume that $\Delta_i'\ge 0$ by choosing $\delta$ small enough. 
Now run an LMMP/$Z_i$ on $K_{X_i'}+\Delta_i'$ 
with scaling of some ample$/Z_i$ divisor. By [\ref{BCHM}], the LMMP 
terminates. By construction, 
$$
K_{X_i'}+\Delta_i'\sim_\R (1-\epsilon)(K_{X_i'}+B_i')/Z_i
$$ 
so the LMMP is also an LMMP$/Z_i$ on $K_{X_i'}+B_i'$. 
\end{rem}

\begin{rem}[On special termination]\label{rem-LMMP-scaling-induction}
Assume that we are given an LMMP with scaling as in Definition \ref{d-LMMP-scaling} 
which consists of only a sequence $X_i\bir X_{i+1}/Z_i$ of log flips, and that $(X_1/Z,B_1)$ 
is $\Q$-factorial dlt. Assume $\rddown{B_1}\neq 0$ and pick a component $S_1$ of $\rddown{B_1}$. 
Let $S_i\subset X_i$ be the birational transform of $S_1$ and $T_i$ the normalisation 
of the image of $S_i$ in $Z_i$. Using standard special termination arguments, 
we will see that termination 
of the LMMP near $S_1$ is reduced to termination  
in lower dimensions. It is well-known that the induced map $S_i\bir S_{i+1}/T_i$ is 
an isomorphism in codimension one if $i\gg 0$ (cf. [\ref{Fujino-st}]). So, we could assume that these maps 
are all isomorphisms in codimension one. Put $K_{S_i}+B_{S_i}:=(K_{X_i}+B_i)|_{S_i}$.
In general, $S_i\bir S_{i+1}/T_i$ is not a $K_{S_i}+B_{S_i}$-flip. To apply induction, 
we need to simplify the situation as follows.

Assume that  $(S_1',B_{S_1'})$ is a $\Q$-factorial dlt blowup 
of $(S_1,B_{S_1})$  (this exists by  
Corollary \ref{c-Q-factorial-blup} below). The idea is to  use 
 the sequence $S_i\bir S_{i+1}/T_i$ to construct an LMMP$/T$ on $K_{S_1'}+B_{S_1'}$
with scaling of $C_{S_1'}$ where $T$ is the normalisation of the image 
of $S_1$ in $Z$ and $C_{S_1'}$ is the pullback of $C_1$. This can be done similar to Remark \ref{rem-lifting-flips} assuming 
that something like $(*)$ is satisfied (in practice, this is satisfied by induction; it also 
can be derived from Theorem \ref{t-mmodel-term-scaling} and this remark is applied only after 
\ref{t-mmodel-term-scaling}). More precisely, we first run an LMMP$/T_1$ on $K_{S_1'}+B_{S_1'}$. 
This is also an LMMP$/T_1$ on $K_{S_1'}+B_{S_1'}$ with scaling of $\lambda_1C_{S_1'}$ because
$K_{S_1'}+B_{S_1'}+\lambda_1C_{S_1'}\equiv 0/T_1$. 
Assuming this terminates, we get a model $S_2'$ on which  $K_{S_2'}+B_{S_2'}$ is nef$/T_1$.
Since $S_1\bir S_{2}$ is 
an isomorphism in codimension one and $K_{S_{2}}+B_{S_{2}}$ 
is ample$/T_1$, $(S_2/T_1,B_{S_2})$ is the lc model of all the pairs $(S_1/T_1,B_{S_1})$, 
$(S_1'/T_1,B_{S_1'})$, and $(S_2'/T_1,B_{S_2'})$. 
Thus, $K_{S_2'}+B_{S_2'}$ is semi-ample$/T_1$ and the map $S_2'\bir S_2$ is a 
morphism which pulls back $K_{S_2}+B_{S_2}$ to $K_{S_2'}+B_{S_2'}$. We continue the 
process by running an LMMP$/T_2$ on $K_{S_2'}+B_{S_2'}$ and so on. So, we get an 
LMMP$/T$ on $K_{S_1'}+B_{S_1'}$
with scaling of $C_{S_1'}$.

If the LMMP$/T$ on $K_{S_1'}+B_{S_1'}$ terminates, then the original LMMP terminates 
near $S_1$. One usually applies this argument to every component of $\rddown{B_1}$ 
to derive termination near $\rddown{B_1}$.
\end{rem}

%%%%%%%%%%%%%%%%%%%%%%%%%%%%%%%%%%%%
%%%%%%%%%%%%%%%%%%%%%%%%%%%%%%%%%%%%
\vspace{0.3cm}
\section{\textbf{LMMP on very exceptional divisors}}

Let $(X/Z,B)$ be a $\Q$-factorial dlt pair such that $X\to Z$ is birational and 
$K_X+B\sim_\R M/Z$ with $M\ge 0$ exceptional$/Z$. Run an LMMP$/Z$ on $K_X+B$ with 
scaling of an ample divisor. In some step of the LMMP, $K_X+B$ becomes nef along very 
general curves of $D/Z$ for any prime divisor $D$. Since $M$ is exceptional$/Z$, 
this is possible only if $M$ is contracted in the process (this follows from 
a general form of the negativity lemma that is discussed below). So, the LMMP terminates.  
This is useful in many situations, for example, to construct a dlt blowup of a lc pair.  
It also plays a crucial role in the proof of Theorem \ref{t-main-under-fg-intro}. 
 
 In this section, we generalise and make precise the above phenomenon. Many of the ideas 
in this section (and in the proof of Theorem \ref{t-main-under-fg-intro}) are actually explicit or implicit in 
Shokurov [\ref{pl-flips}]. However, we would like to give full details here.

\begin{defn}[cf. Shokurov { [\ref{pl-flips}, Definition 3.2]}]\label{d-v-exceptional}
Let $f\colon X\to Y$ be a contraction of normal varieties, $D$ an $\R$-divisor on $X$, and 
$V\subset X$ a closed subset. We say that $V$ is vertical 
over $Y$ if $f(V)$ is a proper subset of $Y$. 
We say that $D$ is very exceptional$/Y$ if $D$ is vertical$/Y$ and  
for any prime divisor $P$ on $Y$ there is a prime divisor $Q$ on $X$ which is 
not a component of $D$ but $f(Q)=P$, i.e. over the generic point of $P$ we have: 
$\Supp f^*P\nsubseteq \Supp D$.  
\end{defn}

If $\codim f(D)\ge 2$, then $D$ is very exceptional. On the other hand, when $f$ is birational, 
then exceptional and very exceptional are equivalent notions. 
The next lemma indicates where one can expect to find very exceptional divisors.

\begin{lem}[cf. Shokurov { [\ref{pl-flips}, Lemma 3.19]}]\label{l-v-exceptional}
Let $f\colon X\to Y$ be a contraction of normal varieties, projective over a normal affine variety 
$Z$. Let $A$ be an ample$/Z$ divisor on $Y$ and $F=f^*A$. If $E\ge 0$ is a divisor on $X$ which is 
vertical$/Y$ and such that $mE=\Fix (mF+mE)$ for every integer $m\gg 0$, then $E$ is very 
exceptional$/Y$. 
\end{lem}
\begin{proof}
We can assume that $A$ is very ample$/Z$. Since $E$ is effective, for each integer $l>0$, 
we have the natural exact sequence 
$0\to \x{O}_X\to \x{O}_X(lE)$ given by the inclusion $\x{O}_X\subseteq  \x{O}_X(lE)$. 
Since $mF=\Mov (mF+lE)$ for each $m\gg 0$, 
the induced homomorphism 
$$
\bigoplus_{m\in\Z} H^0(X,mF)\to \bigoplus_{m\in\Z} H^0(X,mF+lE)
$$ 
of $R(X/Z,F)$-modules is an isomorphism in large degrees.  
This in turn induces a homomorphism  
$$
 \bigoplus_{m\in\Z} H^0(Y,mA)
 \to \bigoplus_{m\in\Z} H^0(Y,(f_*\x{O}_X(lE))(mA))
$$
of $R(Y/Z,A)$-modules which is an isomorphism in large degrees.
Therefore, by [\ref{Hartshorne}, II, Proposition 5.15 and Exercise 5.9], the injective morphism
 $\x{O}_Y=f_*\x{O}_X \to f_*\x{O}_X(lE)$ is also surjective hence we get  
 $\x{O}_Y=f_*\x{O}_X =f_*\x{O}_X(lE)$. 

Assume that $E$ is not very exceptional$/Y$. 
Then, there exist a prime divisor $P$ on $Y$ such that if $Q$ is any prime divisor on $X$ with 
$f(Q)=P$,  then $Q$ is a component of $E$.  
Let $U$ be a smooth open subset of $Y$ 
so that $\x{O}_U\subsetneq \x{O}_U(P|_U)$, $P|_U$ is Cartier, and each component of 
$f^*P|_U$ maps onto $P|_U$. Let $V=f^{-1}U$. 
Then, $\Supp f^*P|_U\subseteq \Supp E|_V$ and for some $l>0$, $f^*P|_U\le lE|_V$. 
But then,  
$$
\x{O}_U\subsetneq \x{O}_U(P|_U)\subseteq f_*\x{O}_V(lE|_V)=\x{O}_U
$$ 
which is a contradiction.
Therefore, $E$ is very exceptional$/Y$.
\end{proof}

Let $S\to Z$ be a projective morphism of varieties and $M$ an $\R$-Cartier divisor on $S$. We say that 
$M$ is \emph{nef on the very general curves of $S/Z$} if there is a countable union $\Lambda$ 
of proper closed subsets of $S$ such that $M\cdot C\ge 0$ for any curve $C$ on $S$ contracted 
over $Z$ satisfying $C\nsubseteq \Lambda$.
We need the following general negativity lemma of Shokurov [\ref{pl-flips}, Lemma 3.22] 
(see also Prokhorov [\ref{Prokhorov}, Lemma 1.7]).

\begin{lem}[Negativity]\label{l-negativity}
Let $f\colon X\to Z$ be a contraction of normal varieties. Let 
$D$ be an $\R$-Cartier divisor on $X$ written as $D=D^+-D^-$ with $D^+,D^-\ge 0$ 
having no common components. Assume that $D^-$   
is very exceptional$/Z$, and that for each component $S$ of $D^-$, $-D|_S$ is
nef on the very general curves of $S/Z$. Then, $D\ge 0$.
\end{lem}
\begin{proof}
Assume that $D^-\neq 0$ otherwise there is nothing to prove. By assumptions, 
there is a countable union $\Lambda$ 
of $\codim \ge 2$ proper closed subsets of $X$ such that $\Lambda\subset \Supp D^-$ 
and such that for any component $S$ of $D^-$ and any curve $C$ of $S/Z$ satisfying 
$C\nsubseteq \Lambda$, we have $-D\cdot C\ge 0$. Let 
$P=f(\Supp D^-)$. By shrinking $Z$ if necessary we can assume that $P$ is 
irreducible and that every component of $D^-$ maps onto $P$. When $\dim P>0$ 
we reduce the problem to lower dimension by taking a hyperplane section of $Z$. 
In contrast, when $\dim P=0$ 
we reduce the problem to lower dimension by taking a hyperplane section of $X$ 
except when $\dim Z=1$ or $\dim Z=\dim X=2$ which can be dealt with directly.

Assume that $\dim P>0$. Let $Z'$ be a very general hyperplane section of $Z$, $X'=f^* Z'$, and let 
$f'$ be the induced contraction $X'\to Z'$. Since $Z'$ is very general, it intersects $P$, and 
$X'$ does not contain any component of $D$ nor any "component" of $\Lambda$. 
So, $\Lambda\cap X'$ is again a countable union of $\codim\ge 2$ proper closed subsets 
of $X'$. Since $X'$ is general, it does not contain any component of the singular 
locus of $X$ hence $D|_{X'}$ is determined on the smooth locus of $X$. Similarly, the negative 
part of $D|_{X'}$ is given by $D^-|_{X'}$ defined on the smooth locus of $X$. Moreover, 
we can assume that $X'$ does not contain any component of $S_1\cap S_2$
if $S_1,S_2$ are prime divisors mapping onto $P$.

We show that the negative part of $D|_{X'}$, say $D|_{X'}^-$, 
is very exceptional$/Z'$: if $\codim P\ge 2$, the claim is trivial; if $\codim P=1$, 
we can assume that $Z$ is smooth at every point of $P$ and that 
every component of $D|_{X'}^-$ maps onto $P':=Z'\cap P$; now since $D^-$ is very 
exceptional$/Z$, there is a component $S$ of $f^*P$ mapping onto $P$ 
such that $S$ is not a component 
of $D^-$; let $Q$ be a component of $S\cap X'$ having codimension one in $X'$ and 
mapping onto $P'$; 
then, by our choice of $X'$, $Q$ is not a component of $D|_{X'}^-$. Since $Q$ 
and each component of $D|_{X'}^-$ maps onto $P'$, and since $Q$ is not a component of $D|_{X'}^-$, 
indeed $D|_{X'}^-$ is very exceptional$/Z'$ as claimed.
On the other hand, if $T$ is any  component of $D|_{X'}^-$, then $T$ is a 
component of $L\cap X'$ for some component $L$ of $D^-$ hence $-D|_T$ is nef 
on the curves of $T/Z'$ outside $\Lambda\cap X'$ (note that by construction 
$T\nsubseteq \Lambda$).
By induction on dimension applied to $D|_{X'}$ and $X'\to Z'$, $D|_{X'}^-=0$ which is a contradiction. 

From now on we can assume that $\dim P=0$.  Assume $\dim Z=1$. 
Then $f^*P$ is a divisor which is numerically zero over $Z$. Let $t$ be the smallest 
real number such that $D+tf^*P\ge 0$. Then, there is a component $S$ of $D^-$ 
which has coefficient zero in $D+tf^*P\ge 0$ but such that $S$ intersects  
 $\Supp (D+tf^*P)$. If $C$ is any curve on $S$ which is not inside 
$S\cap \Supp (D+tf^*P)$ but such that $C$ intersects $\Supp (D+tf^*P)$, then 
$D\cdot C=(D+tf^*P)\cdot C>0$. This contradicts our assumptions.
On the other hand, if $\dim X=\dim Z=2$, the lemma 
is quite well-known and elementary and it can be proved similarly. So, from now on 
we assume that $\dim Z\ge 2$, and $\dim X\ge 3$.

Let $H$ be a very general hyperplane section of $X$ and let $g\colon H\to G$ be the 
contraction given by the Stein factorisation of $H\to Z$. Since $P=f(\Supp D^-)$ is 
just a point and $\dim Z\ge 2$ and $\dim X\ge 3$, it is obvious that $D|_H^-$, 
the negative part of $D|_H$, is 
very exceptional$/G$.  Moreover, arguments similar to the above show that 
if $T$ is any component of $D|_H^-$, then $-D|_H$ is nef on the very general curves of 
$T/G$. So, we can apply induction to deduce that  $D|_H^-=0$ which gives a 
contradiction.
\end{proof}

We now come to the LMMP that was mentioned at the beginning of this section. 
It is a simple consequence of the negativity lemma.

\begin{thm}[=Theorem \ref{t-exceptional-LMMP-intro}]\label{t-exceptional-LMMP}
Let $(X/Z,B)$ be a $\Q$-factorial dlt pair such that $K_X+B\sim_\R M/Z$ with 
$M\ge 0$ very exceptional$/Z$. 
Then, any LMMP$/Z$ on $K_X+B$ with scaling of an ample$/Z$ divisor 
terminates with a model $Y$ on which $K_Y+B_Y\sim_\R M_Y=0/Z$.
\end{thm}
\begin{proof}
Assume that $C\ge 0$ is an ample$/Z$ $\R$-divisor such that $K_X+B+C$ is lc and nef$/Z$. 
Run the LMMP$/Z$ on $K_X+B$ with scaling of $C$. The only divisors that can be 
contracted are the components of $M$ hence $M$ remains very exceptional$/Z$ 
during the LMMP. We may assume that the LMMP consists of only a sequence 
$X_i\bir X_{i+1}/Z_i$ of $K_X+B$-flips$/Z$ with $X_1=X$.
 If $\lambda_i$ are the 
numbers appearing in the LMMP, and $\lambda:=\lim_{i\to \infty}\lambda_i$, then by [\ref{BCHM}], 
we may assume that $\lambda=0$. 
Since  $\lambda=0$, $K_X+B$ is a limit of movable$/Z$ $\R$-Cartier divisors hence 
for any prime divisor $S$ on $X$, $(K_X+B)\cdot \Gamma=M\cdot \Gamma\ge 0$ 
for the very general curves $\Gamma$ of $S/Z$. Now by assumptions $M$ is very exceptional$/Z$ 
hence by the negativity lemma (\ref{l-negativity}), $M\le 0$ which implies that $M=0$. 
So, the LMMP terminates and contracts $M$.
\end{proof}

The same arguments as in the previous theorem imply:

\begin{thm}\label{t-exc-LMMP}
Let $(X/Z,B)$ be a $\Q$-factorial dlt pair such that 
 $X\to Z$ is birational, and  $K_X+B\sim_\R M=M^+-M^-/Z$ where 
$M^+,M^-\ge 0$ have no common components and $M^+$ is exceptional$/Z$. 
Then, any LMMP$/Z$ on $K_X+B$ with scaling of an ample$/Z$ divisor contracts 
$M^+$ after finitely many steps.
\end{thm}

The birationality condition on $X\to Z$ in the previous theorem is needed to make sure 
that  $M^+$ remains very exceptional$/Z$ during the LMMP.

An application of the above theorems is the existence of $\Q$-factorial dlt blowups which is due 
to Hacon (cf. Fujino [\ref{Fujino-dlt-blowup}, Theorem 4.1]). 

\begin{cor}\label{c-Q-factorial-blup}
Let $(X/Z,B)$ be a lc pair. Then, there is a $\Q$-factorial dlt blowup of 
$(X/Z,B)$.
\end{cor}
\begin{proof}
Let $(W/Z,B_W)$ be a log smooth model of $(X/Z,B)$ of type (1) as in Definition \ref{d-log-smooth-model} 
constructed via a log resolution $f\colon W\to X$. Then, 
$K_W+B_W=f^*(K_X+B)+E$ where $E\ge 0$ is exceptional$/X$. Now, 
by Theorem \ref{t-exceptional-LMMP} some LMMP$/X$ on $K_W+B_W$ with scaling of an ample divisor 
terminates with a model $Y$ on which $K_Y+B_Y\sim_\R 0/X$. If $g\colon Y\to X$ is 
the induced morphism, then $K_Y+B_Y=g^*(K_X+B)$ because $E$ is contracted over $Y$. 
Now $(Y/Z,B_Y)$ is the desired $\Q$-factorial dlt blowup.
\end{proof}

\begin{cor}\label{c-wlc-to-mmodel}
Let $(X/Z,B)$ be a lc pair. If $(X/Z,B)$ has a weak lc model, then it has a log minimal model. 
\end{cor}
\begin{proof}
Assume that $(Y'/Z,B_{Y'})$ is a weak lc model of $(X/Z,B)$. Let $(W/Z,B_W)$  
be a  log smooth model of $(X/Z,B)$ of type (1) constructed via a log resolution $f\colon W\to X$.
We can assume that the induced map $g\colon W\bir Y'$ is a morphism.
We can write 
$$
K_W+B_W=f^*(K_X+B)+E
$$
and 
$$
f^*(K_X+B)=g^*(K_{Y'}+B_{Y'})+G
$$
where $E\ge 0$ is exceptional$/X$ and $G\ge 0$ is exceptional$/Y'$ (see Remark \ref{rem-on-m-models}).
Now, $K_W+B_W\sim_\R G+E/Y'$. Moreover, $G+E$ is exceptional$/Y'$: $G$ is exceptional$/Y'$ 
as noted; on the other hand if $D$ is a component of $E$, then $a(D,X,B)>0$ which implies that $D$ is 
exceptional$/Y'$ otherwise $0=a(D,Y',B_{Y'})=a(D,X,B)$; so, $E$ is also exceptional$/Y'$.
 
By Theorem  \ref{t-exceptional-LMMP}, some LMMP$/Y'$ on $K_W+B_W$ with scaling of an ample divisor 
terminates with a model $Y$ on which $K_Y+B_Y\sim_\R 0/Y'$. This means that $K_Y+B_Y$ is the 
pullback of $K_{Y'}+B_{Y'}$ hence $K_Y+B_Y$ is nef$/Z$ and $(Y/Z,B_Y)$ is a log minimal model 
of $(W/Z,B_W)$ hence of $(X/Z,B)$ by Remark \ref{rem-log-smooth-model}. 
\end{proof}

%%%%%%%%%%%%%%%%%%%%%%%%%%%%%%%%%%%%
%%%%%%%%%%%%%%%%%%%%%%%%%%%%%%%%%%%%%
\vspace{0.3cm}
\section{\textbf{From log minimal models to termination with scaling}}

In this section we prove Theorem \ref{t-mmodel-term-scaling-intro} which 
is essentially [\ref{B-II}, Theorem 1.5] but with weaker assumptions.  
The theorem is of independent interest and it will be used repeatedly in subsequent sections.
The proof follows the general idea of deriving termination with scaling from existence 
of log minimal models that was developed in [\ref{BCHM}] and  [\ref{Shokurov-ordered-termination}]. 
In [\ref{BCHM}], finiteness of models is 
used to get termination. However, this works only in the klt case when the boundary is big, 
or in the dlt case when  the boundary contains an ample divisor.
The proof below is closer to [\ref{Shokurov-ordered-termination}] in spirit but we still make 
use of [\ref{BCHM}]. Much of the difficulties in the proof are caused 
by the presence of non-klt singularities.

\begin{thm}[=Theorem \ref{t-mmodel-term-scaling-intro}]\label{t-mmodel-term-scaling}
 Let $(X/Z,B+C)$ be a lc pair of dimension $d$ such that $K_X+B+C$ is nef$/Z$, 
 $B,C\ge 0$ and $C$ is $\R$-Cartier.
Assume that we are given an LMMP$/Z$ on $K_X+B$ with scaling of $C$ as in 
Definition \ref{d-LMMP-scaling} with $\lambda_i$ the corresponding numbers, 
and $\lambda:=\lim_{i\to \infty}\lambda_i$. 
Then, the LMMP terminates in the following cases:

${\rm (i)}$ $(X/Z,B)$ is $\Q$-factorial dlt, $B\ge H\ge 0$ for some ample$/Z$ 
$\R$-divisor $H$,

${\rm (ii)}$ $(X/Z,B)$ is $\Q$-factorial dlt, 
$C\ge H\ge 0$ for some ample$/Z$ $\R$-divisor $H$, and $\lambda>0$,

${\rm (iii)}$ $(X/Z,B+\lambda C)$  has a log minimal model, and $\lambda\neq \lambda_j$ for any $j$.
\end{thm}

The following two lemmas essentially  contain the main points of the proof of Theorem \ref{t-mmodel-term-scaling}. 
We will reduce the proof of Theorem \ref{t-mmodel-term-scaling} 
to these lemmas.
When $(X/Z,B+C)$ 
is $\Q$-factorial klt, then one can easily see that the assumptions of 
Lemma \ref{l-mmodel-term-scaling-2} 
are automatically satisfied, 
so in this case we do not need to go beyond 
the lemmas. 

\begin{lem}\label{l-mmodel-term-scaling}
Theorem \ref{t-mmodel-term-scaling} ${\rm (i)}$ and ${\rm (ii)}$ hold. 
\end{lem}
\begin{proof}
(i)  Since $H$ 
is ample$/Z$, we can perturb the coefficients of $B$ hence assume that 
$(X/Z,B)$ is klt. If $\lambda_i<1$ for some $i$, then 
$({X}/Z,B+\lambda_iC)$ is klt and after finitely many steps we 
could replace $C$ with $\lambda_iC$ hence we could assume that $(X/Z,B+C)$ 
is klt and $B$ is big$/Z$. We can then apply [\ref{BCHM}]. Now, 
assume that $\lambda_i=1$ for every $i$. Then, the LMMP is also an LMMP$/Z$ 
on $K_{X}+B+\frac{1}{2} C$ with scaling of $\frac{1}{2}C$.
By replacing $B$ with $B+\frac{1}{2} C$ and $C$ with $\frac{1}{2}C$,  we can assume that 
every lc centre of $({X}/Z,B+C)$ is inside $\Supp B$. Now perturb 
$B$ again so that $(X/Z,B+C)$ becomes klt and apply [\ref{BCHM}]. 

(ii) The LMMP is also an LMMP$/Z$ on $K_{X}+B+\frac{\lambda}{2} C$
with scaling of $(1-\frac{\lambda}{2})C$. We can replace $B$ with $B+\frac{\lambda}{2} C$, 
$C$ with $(1-\frac{\lambda}{2})C$, and  $\lambda_i$ with 
$\frac{\lambda_i-\frac{\lambda}{2}}{1-\frac{\lambda}{2}}$. After this change, 
 we can assume that $B\ge \frac{\lambda}{2}H$. Now use (i).\\
\end{proof}

\begin{lem}\label{l-mmodel-term-scaling-2}
Theorem \ref{t-mmodel-term-scaling} ${\rm (iii)}$ holds if there is a 
birational map $\phi\colon X\bir Y/Z$ satisfying:

$\bullet$  the induced maps $X_i\bir Y$ are isomorphisms in codimension one for every $i\gg 0$
where $X_i$ is the variety corresponding to $\lambda_i$,

$\bullet$ $(Y/Z,B_Y+\lambda C_Y)$ is a log minimal model of $(X/Z,B+\lambda C)$ 
with respect to the given map $\phi$,

$\bullet$ there is a reduced divisor $A\ge 0$ on $X$ whose components are movable$/Z$ 
and they generate $N^1(X/Z)$,

$\bullet$ $(X/Z,B+C+\epsilon A)$ and 
$(Y/Z,B_Y+\lambda C_Y+\delta C_{Y}+\epsilon A_Y)$ are $\Q$-factorial dlt for some $\delta, \epsilon>0$.

(As usual, for a divisor $D$,  $D_Y$ denotes the birational transform on $Y$.) 
\end{lem}
\begin{proof}
\emph{Step 1.} Assume that the LMMP does not terminate.
Pick $j\gg 0$ so that $\lambda_{j-1}>\lambda_{j}$. Then, $X\bir X_j$ is a 
partial LMMP$/Z$ on $K_X+B+\lambda_j C$. It is also a partial LMMP$/Z$ 
on  $K_X+B+\lambda_j C+\epsilon A$ maybe after choosing a smaller $\epsilon>0$. 
In particular, $(X_j/Z,B_j+\lambda_j C_j+\epsilon A_j)$ is $\Q$-factorial dlt.
Now we can replace $(X/Z,B+C)$ with $(X_j/Z,B_j+\lambda_j C_j)$ hence assume that  
the LMMP consists of only a sequence $X_i\bir X_{i+1}/Z_i$ of log flips$/Z$ 
starting with $(X_1/Z,B_1)=(X/Z,B)$.
Moreover, by replacing $B$ with $B+\lambda C$ we may also assume that $\lambda=0$.

We will reduce the problem to the situation in which not only $K_Y+B_Y$ 
is nef$/Z$ but also $K_Y+B_Y+\lambda_i C_Y$ is nef$/Z$ for some $i$. 
After that some simple calculations allow us to show that the LMMP terminates (see 
Step 5). 

\emph{Step 2.} We will show that, perhaps after replacing $\delta$ and $\epsilon$ with 
smaller positive numbers, for any number $\delta'$ and $\R$-divisor $A_Y'$ satisfying $0\le \delta'\le \delta$ and $0\le A_Y'\le \epsilon A_Y$, any 
 LMMP$/Z$ on $K_Y+B_Y+\delta' C_Y+A_Y'$ consists of only a sequence of log flips which are flops with respect to $(Y/Z,B_Y)$ (that is, $K_Y+B_Y$ is numerically trivial on each extremal ray that is contacted in the process). 
First we show that $K_Y+B_Y+\delta' C_Y+A_Y'$ is a limit of movable$/Z$ $\R$-divisors.  
Since $\delta$ is sufficiently small, we can assume that  
$\lambda_{i-1}\ge \delta'\ge \lambda_i$ for some $i$.  
By definition of LMMP with scaling, $K_{X_i}+B_i+\lambda_{i-1} C_i$ 
and $K_{X_i}+B_i+\lambda_{i} C_i$ are both nef$/Z$ hence $K_{X_i}+B_i+\delta' C_i$ 
is also nef$/Z$. Thus, 
$K_{X_i}+B_i+\delta' C_i$ is (numerically) a limit of movable$/Z$ $\R$-divisors which in turn 
implies that  $K_Y+B_Y+\delta' C_Y$ is also a limit of movable$/Z$ $\R$-divisors. 
On the other hand, each component of $A_Y$ is movable$/Z$ so $A_Y'$ is 
a movable$/Z$ $\R$-divisor. Therefore, $K_Y+B_Y+\delta' C_Y+A_Y'$ is a limit of movable$/Z$ $\R$-divisors. 
This implies that any 
 LMMP$/Z$ on $K_Y+B_Y+\delta' C_Y+A_Y'$ consists of only a sequence of log flips. 
Finally, since $K_Y+B_Y$ is nef$/Z$, by [\ref{B-II}, Proposition 3.2], 
perhaps after replacing $\delta$ and $\epsilon$ with 
smaller positive numbers, any step of such an LMMP would be a flop with respect to $(Y/Z,B_Y)$. 

\emph{Step 3.}
 Fix some $i\gg 0$ so that $\lambda_i<\delta$. 
Since $X\bir X_i$ is a 
partial LMMP$/Z$ on $K_X+B+\lambda_i C$, 
 there is $0<\tau < \epsilon$ 
such that $(X_i/Z,B_i+\lambda_i C_i+\tau A_i)$ is dlt. 
Run an LMMP$/Z$ 
on  $K_{X_i}+B_i+\lambda_i C_i+\tau A_i$ with scaling of some ample$/Z$ divisor.
For reasons as in Step 2, 
we can assume that this LMMP consists of only a  
sequence of log flips which are flops with respect to $(X_i/Z,B_i+\lambda_i C_i)$. 
On the other hand, since the components of $A_i$ generate $N^1(X_i/Z)$, we can assume that there is an ample$/Z$ $\R$-divisor 
$H\ge 0$ such that $\tau A_i \equiv H+H'/Z$ where $H'\ge 0$ and 
$$
(X_i/Z,B_i+\lambda_i C_i+H+H')
$$
is dlt. The LMMP on $K_{X_i}+B_i+\lambda_i C_i+\tau A_i$ is also an LMMP on 
$$
K_{X_i}+B_i+\lambda_i C_i+H+H'
$$ hence it 
terminates by Lemma \ref{l-mmodel-term-scaling}
and we get a model $T$ on which both 

$$
K_T+B_T+\lambda_i C_T 
~~~~\mbox{and}~~~~ 
K_T+B_T+\lambda_i C_T+\tau A_T
$$ 
are nef$/Z$.  Again since the components of $A_T$ generate $N^1(T/Z)$,  
there is $0\le A'_T\le \tau A_T$ so that $K_T+B_T+\lambda_i C_T+ A_T'$ is ample$/Z$ and $\Supp A_T'=\Supp A_T$.

\emph{Step 4.} Now run an LMMP$/Z$ on 
$K_Y+B_Y+\lambda_i C_Y+A'_Y$ with scaling of some ample$/Z$ divisor  where $A_Y'$ 
is the birational tranform of $A_T'$. By Step  2, the LMMP
consists of only a sequence of log flips which are flops with respect to $(Y/Z,B_Y)$, 
and it terminates for reasons similar to Step 3. Actually, the LMMP terminates on 
 $T$ because $K_T+B_T+\lambda_i C_T+ A_T'$ is ample$/Z$.  So, by replacing $Y$ with $T$ we 
 could from now on assume that $K_Y+B_Y+\lambda_i C_Y$ 
is nef$/Z$. In particular, $K_Y+B_Y+\lambda_j C_Y$ is nef$/Z$ for any $j\ge i$ since $\lambda_j\le \lambda_i$ 
and since $K_Y+B_Y$ is nef$/Z$. 

\emph{Step 5.} Pick $j> i$ so that $\lambda_j<\lambda_{j-1}$ and let $r\colon U\to X_j$ and $s\colon U\to Y$ 
be a common resolution. Then,  
\begin{equation*}
\begin{split}
r^*(K_{X_j}+B_j+\lambda_j C_j) & = s^*(K_Y+B_Y+\lambda_j C_Y),\\ 
r^*(K_{X_j}+B_j) & \gneq s^*(K_Y+B_Y),\\
r^*C_j & \lneq  s^* C_Y
\end{split}
\end{equation*}
where the first equality holds because both $K_{X_j}+B_j+\lambda_j C_j$ and $K_Y+B_Y+\lambda_j C_Y$ 
are nef$/Z$ and $X_j$ and $Y$ are isomorphic in codimension one, the second inequality holds 
by Remark \ref{rem-on-m-models} because $K_Y+B_Y$ is nef$/Z$ but $K_{X_j}+B_j$ is not nef$/Z$, and the third claim follows from the other two. Now
\begin{equation*}
\begin{split}
r^*(K_{X_j}+B_j+\lambda_{j-1} C_j) 
 & = r^*(K_{X_j}+B_j+\lambda_{j} C_j)+ r^*(\lambda_{j-1}-\lambda_j)C_j\\
& \lneq  s^*(K_Y+B_Y+\lambda_{j} C_Y)+s^*(\lambda_{j-1}-\lambda_j)C_Y\\
& =  s^*(K_Y+B_Y+\lambda_{j-1} C_Y)
\end{split}
\end{equation*}

However, since $K_{X_j}+B_j+\lambda_{j-1} C_j$ 
and  $K_Y+B_Y+\lambda_{j-1} C_Y$ are both nef$/Z$, we have 
$$
r^*(K_{X_j}+B_j+\lambda_{j-1} C_j) =s^*(K_Y+B_Y+\lambda_{j-1} C_Y)
$$ 
This is a contradiction and the sequence of log flips terminates as claimed.
\end{proof}

The reader may want to have a look at Remark \ref{rem-lifting-flips} before reading the 
rest of this section.

\begin{proof}(of Theorem \ref{t-mmodel-term-scaling})
\emph{Step 1.} In view of Lemma \ref{l-mmodel-term-scaling} it is enough to 
treat the case ${\rm (iii)}$. We can replace $B$ with $B+\lambda C$ hence 
assume that $\lambda=0$. Moreover, we may assume that the LMMP consists of 
only a sequence $X_i\bir X_{i+1}/Z_i$ of log flips starting with 
$(X_1/Z,B_1)=(X/Z,B)$. Pick $i$ so that $\lambda_i>\lambda_{i+1}$. Then, $\Supp C_{i+1}$ does not 
contain any lc centre of $(X_{i+1}/Z,B_{i+1}+\lambda_{i+1}C_{i+1})$ because  $(X_{i+1}/Z,B_{i+1}+\lambda_{i}C_{i+1})$ is 
lc. Thus, by replacing $(X/Z,B)$ with $(X_{i+1}/Z,B_{i+1})$ 
and $C$ with $\lambda_{i+1}C_{i+1}$ we may 
assume that no lc centre of $(X/Z,B+C)$ is inside $\Supp C$. Moreover, since 
there are finitely many lc centres of $(X/Z,B)$, perhaps after truncating the 
sequence, we can assume that no lc centre is contracted in the 
sequence. We will reduce the problem to the situation of Lemma \ref{l-mmodel-term-scaling-2}.

\emph{Step 2.} 
 By assumptions, there is a log minimal model $(Y/Z,B_Y)$ for $(X/Z,B)$.
Let $\phi \colon X\bir Y/Z$ be the corresponding birational map.
Let $f\colon W\to X$ and $g\colon W\to Y$ 
be a common log resolution of $(X/Z,B+C)$ and $(Y/Z,B_Y+C_Y)$ where $C_Y$ is the birational transform of $C$. 
By Remark \ref{rem-on-m-models}, 
$$
E:=f^*(K_X+B)-g^*(K_Y+B_Y)
$$
is effective and exceptional$/Y$.
Let $B_W$ be the birational transform of $B$ plus the reduced exceptional divisor of $f$, 
and let $C_W$ be the 
birational transform of $C$. Pick a sufficiently small $\delta\ge 0$.
Since $(X/Z,B)$ is lc,
$$
E':=K_W+B_W-f^*(K_X+B)
$$
is effective and exceptional$/X$. Actually, $E'$ is also exceptional$/Y$ 
because if $D$ is a component of $E'$ which is not exceptional$/Y$, then 
$$
0=a(D,Y,B_Y)\ge a(D,X,B)\ge 0
$$ 
which is a contradiction since $a(D,X,B)=0$ means that $D$ cannot be a component of $E'$. 

On the other hand, since $Y$ is $\Q$-factorial, there is 
an exceptional$/Y$ $\R$-divisor $F$ on $W$ such that  $C_W+F\equiv 0/Y$.
So, 
$$
K_W+B_W+\delta C_W\equiv E+E'+\delta C_W\equiv E+E'-\delta F/Y
$$
and since $\delta$ is sufficiently small, the support of $E+E'$ is contained in the 
 support of the effective part of $E+E'-\delta F$. 
Now by Theorem \ref{t-exc-LMMP}, if we run an LMMP$/Y$ on $K_W+B_W+\delta C_W$ with scaling of an ample  
divisor, then we arrive at a model $V$ on which $E_{V}+E'_{V}=0$, that is, 
$E+E'$ is contracted over $V$.

\emph{Step 3.} 
We prove that $\phi\colon X\bir Y$ does not contract any divisors. 
Assume otherwise and let 
$D$ be a prime divisor on $X$ contracted by $\phi$. Then $D^\sim$ the birational transform of 
$D$ on $W$ is a component of $E$ because by definition of log minimal models $a(D,X,B)<a(D,Y,B_Y)$.
Now, in Step 2 take $\delta=0$. The LMMP contracts $D^\sim$  
since $D^\sim$ is a component of $E$ and $E$ is contracted. This is not possible as 
we can see as follows. Since $({X}/Z,B+\lambda_i C)$ is lc,
$$
K_W+B_W+\lambda_i C_W-f^*(K_{X}+B+\lambda_i C)
$$
is effective and exceptional$/X$. On the other hand, $({X_i}/Z,B_i+\lambda_i C_i)$ 
is a weak lc model of $({X}/Z,B+\lambda_i C)$ hence 
$f^*(K_{X}+B+\lambda_i C)\ge M$
where $M$ is the pullback of $K_{X_i}+B_i+\lambda_i C_i$ under $W\bir X_i$, that is, 
$$
M=p_*q^*(K_{X_i}+B_i+\lambda_i C_i)
$$ 
for some common resolution $p\colon W'\to W$ and 
$q\colon W'\to X_i$.
Therefore, 
$$
K_W+B_W+\lambda_i C_W= M+G
$$ 
where $G$ is effective and exceptional$/X$. Since $K_{X_i}+B_i+\lambda_i C_i$ 
is nef$/Z$, it is (numerically) a limit of movable$/Z$ $\R$-divisors hence 
$M$ is a limit of movable$/Z$ $\R$-divisors. If $\lambda_i$ is sufficiently small 
then $W\bir V$ (of Step 2) is a partial LMMP$/Z$ on $K_W+B_W+\lambda_i C_W$.
But since $G$ is exceptional$/X$ and since $D$ is a divisor on $X$, $D^\sim$ is not a component of 
$G$ hence $D^\sim$ cannot be contracted over $V$ by the LMMP of Step 2, a contradiction. Thus $\phi$ does not contract divisors.

\emph{Step 4.} We will construct a $\Q$-factorial dlt blowup of $(Y/Z,B_Y)$ as follows. 
 In Step 2, take $\delta>0$ which is sufficiently small by assumptions.  
Let $Y'$ be the model $V$ obtained. Since $X\bir Y$ does not contract divisors, 
by construction, each prime exceptional$/Y$  divisor on $Y'$ appears with 
coefficient one in $B_{Y'}$.  So, in view of 
$$
K_{Y'}+B_{Y'}\equiv E_{Y'}+E'_{Y'}=0/Y
$$  
we deduce that $K_{Y'}+B_{Y'}$ is the pullback of $K_{Y}+B_{Y}$ and that 
$(Y'/Z,B_{Y'})$ is a $\Q$-factorial dlt blowup of $(Y/Z,B_Y)$. 
Moreover, $(Y'/Z,B_{Y'}+\delta C_{Y'})$ is dlt.

\emph{Step 5.} We will construct a $\Q$-factorial dlt blowup of $(X/Z,B)$ as follows. 
Since $(X/Z,B+C)$ is lc,
$$
E'':=K_W+B_W+C_W-f^*(K_X+B+C)
$$
is effective and exceptional$/X$.
Run an LMMP$/X$ on  
$K_W+B_W+C_W$ with scaling of an ample$/X$ divisor. Since  
$K_W+B_W+C_W\equiv E''/X$, by Theorem \ref{t-exceptional-LMMP}, the LMMP 
terminates on a model $X'$. 
In particular, $(X'/Z,B'+C')$ is a $\Q$-factorial dlt blowup of $(X/Z,B+C)$ 
where $B'$ is the pushdown of $B_W$ and $C'$ is the pushdown of $C_W$. 
Moreover, the LMMP does not contract an exceptional prime divisor $D$ of $W\to X$
iff $a(D,X,B+C)=0$. 
Since $\Supp C$ does not contain any lc centre of $(X/Z,B+C)$ by Step 1, 
 the LMMP does not contract an exceptional prime divisor $D$ of $W\to X$
iff $a(D,X,B)=0$. Therefore, 
$(X'/Z,B')$ is a $\Q$-factorial dlt blowup of $(X/Z,B)$ and $C'$ is the 
pullback of $C$. 
Note that the prime exceptional divisors of  $\phi^{-1}$ 
are not contracted$/X'$ since their log discrepancies with respect to $(X/Z,B)$ are all $0$.

\emph{Step 6.} 
By Remark \ref{rem-lifting-flips} (1), we can lift the sequence $X_i\bir X_{i+1}/Z_i$ 
to an LMMP$/Z$ on $K_{X'}+B'$ with scaling of $C'$ if the property $(*)$ 
of the remark holds. If $(*)$ holds, then continue from the next paragraph. 
But if $(*)$ does not hold for some $i$, then we can replace  our sequence $X_i\bir X_{i+1}/Z_i$ 
with the LMMP in $(*)$ and repeat Steps 1-5 again. In particular, we could assume that 
each $(X_i/Z,B_i)$ is $\Q$-factorial dlt. But in that case, 
by Remark \ref{rem-lifting-flips} (2), $(*)$ holds so we will not need to 
deal with $(*)$ again.

From now on we assume that we have lifted $X_i\bir X_{i+1}/Z_i$ 
 to  an LMMP$/Z$ on $K_{X'}+B'$ with scaling of $C'$. 
We show that $Y'\bir X'$ does not contract divisors. 
Suppose that $D$ is a prime divisor on $Y'$ which is exceptional$/X'$. 
If $D$ is not exceptional$/Y$, then $D$ is an exceptional divisor of $\phi^{-1}$
and this contradicts the last sentence of Step 5. Thus, $D$ is exceptional$/Y$. But since 
$(Y'/Z,B_{Y'})$ is a $\Q$-factorial dlt blowup of $(Y/Z,B_Y)$, $a(D,Y,B_Y)=0$ which 
in turn implies that $a(D,X,B)=0$ and again Step 5 gives a contradiction. 

 We show that $X'\bir Y'$ also does not contract divisors. By Step 3, it is 
enough to show that $(Y'/Z,B_{Y'})$ is a log minimal model of $(X'/Z,B')$. 
By construction, $B_{Y'}$ is the pushdown of $B'$. So, it remains to 
 compare log discrepancies. Assume that $D$ is a 
prime divisor on $X'$ which is exceptional$/Y'$. 
Since $X\bir Y$ does not contract divisors by Step 3, $D$ is exceptional$/X$. 
In particular, $a(D,X',B')=a(D,X,B)=0$. If $a(D,Y',B_{Y'})=0$, then $a(D,Y,B_Y)=0$ hence   
the birational transform of 
$D$ cannot be a component of $E+E'+\delta C_W$ in Step 2 so it could not be contracted 
over $Y'$ which is a contradiction. Therefore,  $a(D,Y',B_{Y'})>0$. Thus, 
for every prime exceptional divisor of $X'\bir Y'$ we have shown that 
 $a(D,X',B')<a(D,Y',B_{Y'})$ which implies that 
 $(Y'/Z,B_{Y'})$ is a log minimal model of $(X'/Z,B')$ by construction.

 \emph{Step 7.} 
 Let $A\ge 0$ be a reduced divisor on $W$ whose components 
are general ample$/Z$ divisors such that they generate $N^1(W/Z)$. By Step 5, $X'$ is obtained 
by running some LMMP on 
$K_W+B_W+C_W$. This LMMP is a partial LMMP on $K_W+B_W+C_W+\epsilon A$ 
for any sufficiently small $\epsilon>0$. In particular, $(X'/Z,B'+C'+\epsilon A')$ is dlt 
where $A'$ is the birational transform of $A$. 
For similar reasons, we can choose $\epsilon$ so that 
$(Y'/Z,B_{Y'}+\delta C_{Y'}+\epsilon A_{Y'})$ is also dlt. 
Now apply Lemma \ref{l-mmodel-term-scaling-2} to the LMMP$/Z$ on $K_{X'}+B'$ with scaling of $C'$ 
of Step 6. 
\end{proof}

%%%%%%%%%%%%%%%%%%%%%%%%%%%%%%%%%%%%
%%%%%%%%%%%%%%%%%%%%%%%%%%%%%%%%%%%%
\vspace{0.3cm}
\section{\textbf{Shokurov bss-ampleness, finite generation, and the klt case}}

The next theorem shows that  
more general versions of Theorem \ref{conj-main-general} and Theorem 
\ref{t-acc-main-general-modified} hold when $(X/Z,B)$ has a 
finitely generated algebra $\mathcal{R}(X/Z,K_X+B)$. The result is an easy consequence of  
Lemma \ref{l-v-exceptional} and Lemma \ref{l-negativity} and it should be 
considered as a very special case of Shokurov's attempt in relating 
bss-ampleness and finite generation [\ref{pl-flips}, Theorem 3.18](see also 
[\ref{Lai}][\ref{Fujino-dlt-blowup}] for more recent adaptations). Though we are not using  
Shokurov's terminology of bss-ample divisors but the next theorem is saying that 
finite generation implies bss-ampleness in the specific situation we are 
concerned with. Shokurov proves that in general finite generation together 
with the so-called global almost generation property implies bss-ampleness.

\begin{thm}(=Theorem \ref{t-main-under-fg-intro})\label{t-main-under-fg}
Let $(X/Z,B)$ be a $\Q$-factorial dlt pair where $B$ is a $\Q$-divisor and $f\colon X\to Z$ is surjective. 
Assume further that $\mathcal{R}(X/Z,K_X+B)$ is a finitely generated $\mathcal{O}_Z$-algebra, 
and that $(K_X+B)|_{X_\eta}\sim_\Q 0$ where $X_\eta$ is the generic fibre of $f$. 
Then, any LMMP$/Z$ on $K_X+B$ with scaling of an ample$/Z$ divisor 
terminates with a good log minimal model. 
\end{thm}
\begin{proof}
We can assume that $f$ is a contraction.
Run an LMMP$/Z$ on $K_X+B$ with scaling of some ample$/Z$  divisor. Since 
termination and semi-ampleness$/Z$ are both local on $Z$, we can assume that 
$Z$ is affine, say $\Spec R$. 
By Theorem \ref{t-mmodel-term-scaling}, it is enough to show that 
$(X/Z,B)$ has a good log minimal model.
Let $I$ be a positive integer so that $I(K_X+B)$ is Cartier. 
Since ${R}(X/Z,K_X+B)$ is a finitely generated $R$-algebra, 
perhaps after replacing $I$ with a multiple, there exist a log resolution 
$g\colon W\to X$, a divisor $E\ge 0$ and a free divisor $F$ on $W$ such 
that 
$$
\Fix g^*mI(K_X+B)=mE ~~~~ \mbox{and} ~~~~ \Mov g^*mI(K_X+B)=mF
$$ 
for every $m>0$ (cf. [\ref{pl-flips}] or [\ref{B-dam}, Theorem 4.3]).
Let $h\colon W\to T/Z$ be the contraction defined by $|F|$.
Then, since $(K_X+B)|_{X_\eta}\sim_\Q 0$, the map $T\to Z$ is birational 
and $E$ is vertical$/T$.

Choose a boundary $B_W$ so that $(W/Z,B_W)$ is a log smooth model of $(X/Z,B)$ of type (1)
as in Definition \ref{d-log-smooth-model}, that is, $B_W$ is the birational transform of 
$B$ plus the reduced exceptional divisor of $g$.
We can write 
$$
I(K_W+B_W)=g^*I(K_X+B)+E'
$$
 where $E'\ge 0$ is 
exceptional$/X$.
So, 
$$
\Fix mI(K_W+B_W)=mE+mE'~~~ \mbox{and}~~~\Mov mI(K_W+B_W)=mF
$$ 
Run the LMMP$/T$ on $K_W+B_W$ with scaling of some ample$/T$ divisor. Since 
$(K_X+B)|_{X_\eta}\sim_\Q 0$, there is a non-empty open subset $U\subseteq Z$ so that $K_X+B\sim_\Q 0$ 
over $U$ and $T\to Z$ is an isomorphism over $U$. 
So, over $U$, $(X/Z,B)$ is a weak lc model of $(W/Z,B_W)$ hence by Corollary \ref{c-wlc-to-mmodel}, 
 $(W/Z,B_W)$ has a log minimal model over $U$ (which is just a suitable $\Q$-factorial 
dlt blowup of $(X/Z,B)$). 
Thus, by Theorem \ref{t-mmodel-term-scaling}, the LMMP terminates over $U$. By construction, over $U$ 
we have
$$
I(K_W+B_W)\sim E+E'+F\sim E+E'
$$
so we reach a model $Y'$ on which $E_{Y'}+E_{Y'}'\sim_\Q 0$ over $U$, 
in particular, $E_{Y'}+E_{Y'}'$ is vertical$/T$. On the other hand, since $W\bir Y'$ is a 
partial LMMP$/T$ on $K_W+B_W$, 
$$
\Fix mI(K_{Y'}+B_{Y'})=\Fix (mE_{Y'}+mE_{Y'}'+mF_{Y'})=mE_{Y'}+mE_{Y'}'
$$
hence by Lemma \ref{l-v-exceptional}, $E_{Y'}+E_{Y'}'$ is very exceptional$/T$.  Now, by Theorem \ref{t-exceptional-LMMP}, 
there is an LMMP$/T$ on $K_{Y'}+B_{Y'}$ which ends up with a 
 log minimal model $(Y/T,B_Y)$ on which 
$$
I(K_Y+B_Y)\sim_\Q E_{Y}+E_{Y}'=0/T
$$
 In particular, 
 $I(K_Y+B_Y)\sim_\Q F_Y/Z$. Thus, $K_Y+B_Y$ is semi-ample$/Z$ 
 and $(Y/Z,B_Y)$ is a good log minimal model of $(W/Z,B_W)$ hence of $(X/Z,B)$ by Remark \ref{rem-log-smooth-model}.
\end{proof}

\begin{thm}\label{t-main-klt} 
We have the following for klt pairs:

$\rm (1)$ Let $(X/Z,B)$ be a $\Q$-factorial klt pair where $B$ is a $\Q$-divisor and $f\colon X\to Z$ is surjective. 
Assume that $(K_X+B)|_{X_\eta}\sim_\Q 0$ where $X_\eta$ is the generic fibre of $f$. 
Then, any LMMP$/Z$ on $K_X+B$ with scaling of an ample$/Z$ divisor 
terminates with a good log minimal model.
 
$\rm (2)$ Theorem \ref{conj-main-general} and Theorem \ref{t-acc-main-general-modified} hold when $(X/Z,B)$ is klt.
\end{thm}
\begin{proof}  
(1) By [\ref{BCHM}], 
$\mathcal{R}(X/Z,K_X+B)$ is a finitely generated  $\mathcal{O}_Z$-algebra so  
we can apply Theorem \ref{t-main-under-fg}.
 
(2) Suppose that $(X/Z,B)$ is klt under the assumptions of Theorem \ref{conj-main-general}.
By taking a $\Q$-factorial dlt blowup as in Corollary \ref{c-Q-factorial-blup} 
we can assume that $X$ is $\Q$-factorial.
If $A$ is not vertical$/Z$, then $K_X+B$ is not pseudo-effective$/Z$ hence 
$(X/Z,B)$ has a Mori fibre space by [\ref{BCHM}].
So, we can assume that $A$ is vertical$/Z$. Then, 
 $(K_X+B)|_{X_\eta}\sim_\Q 0$ where $X_\eta$ is the generic fibre of $f$. 
Now use (1). We can treat Theorem \ref{t-acc-main-general-modified} in a similar way. 
Note that here we do not need ACC for lc thresholds.
\end{proof}

Though the last theorem settles the klt case of Theorem \ref{conj-main-general} 
and Theorem \ref{t-acc-main-general-modified} but 
we prove further results in this direction as we will need them to deal 
with the lc case (eg, proof of \ref{l-LMMP-scaling}). 
The next two theorems unfortunately do not simply 
follow from Theorem \ref{t-main-klt} because the boundaries that appear 
are not necessarily rational.

\begin{thm}\label{t-main-klt-trivial}
Let $(X/Z,B)$ be a klt pair where $f\colon X\to Z$ is surjective. 
Assume further that there is a contraction $g\colon X\to S/Z$ such that 
 $K_X+B\sim_\R 0/S$ and that $S\to Z$ is generically finite. 
 Then, $(X/Z,B)$ has a good log minimal model. (Note that 
 $B$ is not necessarily a $\Q$-divisor.) 
\end{thm}
\begin{proof}
As mentioned above we cannot apply Theorem \ref{t-main-klt} because $B$ 
is not assumed to be rational. However, the proof given below is somewhat similar to the 
proof of Theorem \ref{t-main-under-fg}.

We can assume that $f$ is a contraction hence $S\to Z$ is birational. 
Moreover, by taking a $\Q$-factorial dlt blowup as in Corollary \ref{c-Q-factorial-blup} 
we can assume that $X$ is $\Q$-factorial.
Since $K_{X}+B\sim_\R 0/S$, by Ambro [\ref{Ambro-adjunction}] and Fujino-Gongyo 
[\ref{Fujino-Gongyo-adjunction}, Theorem 3.1], there is a 
boundary $B_S$ on $S$ such that 
$$
K_{X}+B\sim_\R g^*(K_{S}+B_{S})
$$
and such that $(S/Z,B_{S})$ is klt and of general type. So, by [\ref{BCHM}], 
there is a weak lc model $(T/Z,B_{T})$ for $(S/Z,B_{S})$ obtained by running 
an LMMP$/Z$ on $K_S+B_S$ with scaling of some ample$/Z$ divisor. Moreover, 
since $T\bir S$ does not contract divisors, there are non-empty open subsets 
$U\subseteq S$ and $V\subseteq T$ such that the induced map 
$U\bir V$ is an isomorphism and such that $\codim (T\setminus V)\ge 2$.

Take a log smooth model $(W/Z,B_W)$ of $(X/Z,B)$ of type (2) as in Definition \ref{d-log-smooth-model}
using a log resolution  $e\colon W\to X$.  
Then, $(W/Z,B_W)$ is klt and 
$$
K_W+B_W=e^*(K_{X}+B)+E
$$
 where $E\ge 0$ is exceptional$/X$, and 
 any log minimal model of $(W/Z,B_W)$ is also a log minimal model of 
$(X/Z,B)$ by Remark \ref{rem-log-smooth-model}. Moreover, we can assume that the induced maps 
$\phi\colon W\bir S$ and $\psi\colon W\bir T$ are both morphisms. 
Run an LMMP$/T$ on $K_W+B_W$ with scaling of some ample$/T$ divisor. 
Since $K_X+B\sim_\R 0$ over $U$, $(X/Z,B)$ is a log minimal model of $(W/Z,B)$ over $U=V$ 
hence by Theorem \ref{t-mmodel-term-scaling} the LMMP terminates over 
$V$: we reach a model $Y'$ such that $K_{Y'}+B_{Y'}\sim_\R 0$ over $V$, and the remaining steps of the LMMP 
contract extremal rays only over $T\setminus V$.  

Let $\pi$ be the induced map $W\bir Y'$. 
Let $M\subseteq X$, $N\subseteq W$, and $O\subseteq Y'$ be the inverse images of $U=V$. 
Then, as mentioned $(M/V,B|_M)$ is a log minimal model of $(N/V,B_W|_N)$ hence $M$ and $O$ 
are isomorphic in codimension one. Thus, $E|_N$ is contracted over $O$. 
 So, in particular, $E_{Y'}=\pi_*E$ is a divisor on 
$Y'$ which is mapped into $T\setminus V$ hence it is very exceptional$/T$. 
On the other hand, since $(T/Z,B_{T})$ is a weak lc model of $(S/Z,B_{S})$,
$$
\phi^*(K_{S}+B_{S})=\psi^*(K_{T}+B_{T})+G
$$
 where $G\ge 0$ is very 
exceptional$/T$ because $G$ is mapped into $T\setminus V$. Thus,

\begin{equation*}
\begin{split}
 K_{Y'}+B_{Y'}  =\pi_*(K_{W}+B_{W})&=\pi_*(e^*(K_{X}+B)+E)\\
& =\pi_*e^*(K_{X}+B)+\pi_*E\\
& \sim_\R \pi_*\phi^*(K_{S}+B_{S})+\pi_*E\\
& =\pi_*\psi^*(K_{T}+B_{T})+\pi_*G+\pi_*E
\end{split}
\end{equation*}

By construction, $\pi_*G+\pi_*E\ge 0$ is very exceptional$/T$. So, by 
Theorem \ref{t-exceptional-LMMP}, there is an LMMP$/T$ on $K_{Y'}+B_{Y'}$ 
with scaling of some ample$/T$ divisor which ends up with a model $Y$ 
on which $K_Y+B_Y\sim_\R 0/T$ hence 
$$
K_Y+B_Y\sim_\R \nu^* (K_{T}+B_{T})
$$ 
where $\nu\colon Y\to T$ is the induced morphism. 
 In particular, $K_Y+B_Y$ is semi-ample$/Z$ as $K_T+B_T$ is semi-ample$/Z$.
 So, $(Y/Z,B_Y)$ is a good log minimal model of $(W/Z,B_W)$ hence of $(X/Z,B)$.
\end{proof}

\begin{thm}\label{t-main-klt-termination}
Let $(X/Z,B)$ be a klt pair where $B$ is a $\Q$-divisor and $f\colon X\to Z$ is surjective. 
Assume further that  $(K_X+B)|_F\sim_\Q 0$ for the generic fibre $F$ of $f$. 
If $K_X+B+H$ is klt and nef$/Z$ for some $\Q$-divisor $H\ge 0$, then the LMMP$/Z$ on $K_X+B$ with 
scaling of $H$ terminates if either

$(1)$ $B$ is big$/Z$ or $H$ is big$/Z$, or 

$(2)$ $H$ is vertical$/Z$ and $\lambda\neq \lambda_j$ for any $j$ where $\lambda_j$ are the 
numbers appearing in the LMMP with scaling (as in Definition \ref{d-LMMP-scaling}) and 
$\lambda=\lim_{i\to \infty} \lambda_i$.
\end{thm}
\begin{proof}
We can assume that $f$ is a contraction.
By Theorem \ref{t-main-klt}, $(X/Z,B)$ has a good log minimal model.
Assume that condition (1) holds. If $B$ is big$/Z$, the termination follows from 
[\ref{BCHM}]. If $H$ is big$/Z$, and if $\lambda>0$, then again the termination follows 
from [\ref{BCHM}] as we can replace $B$ with $B+\lambda H$. If $\lambda=0$,  use 
Theorem \ref{t-mmodel-term-scaling}.

Now assume that condition (2) holds. Since $B,H$ are $\Q$-divisors,
each $\lambda_i$ is rational. However, at this stage we do not know whether 
$\lambda$ is rational. So, we 
cannot apply Theorem \ref{t-main-klt} to get a log minimal model of $(X/Z,B+\lambda H)$ 
to allow us to apply Theorem \ref{t-mmodel-term-scaling}.  
  
We may assume that the LMMP consists of only a sequence $X_i\bir X_{i+1}/Z_i$ of 
log flips and that $X_1=X$. Moreover, by Remark \ref{rem-lifting-flips} (1)(2), we can lift 
the sequence to the $\Q$-factorial situation hence assume that $X_i$ are $\Q$-factorial. 
Now by Theorem \ref{t-main-klt}, $({X}/Z,B+\lambda_{1}H)$ has a good log minimal model. 
But then every log minimal model of $({X}/Z,B+\lambda_{1}H)$ is good by Remark \ref{rem-two-wlc-models}. 
Since 
$K_{X}+B+\lambda_{1}H$ is nef$/Z$, $({X}/Z,B+\lambda_{1}H)$ is a log minimal model of 
itself hence $K_{X}+B+\lambda_{1}H$ is semi-ample$/Z$. Let 
$g\colon X\to T$ be the contraction associated to it. Since 
$H$ is vertical$/Z$ and $(K_X+B)|_F\sim_\Q 0$, the map $T\to Z$ is birational so the generic fibre of $g$ 
and $f$ are the same. By Theorem \ref{t-main-klt}, some LMMP$/T$ 
on $K_X+B$ terminates with a good log minimal $(Y/T,B_Y)$ 
of $(X/T,B)$. 
Let $h \colon Y\to S/T$ be the 
contraction associated to $K_Y+B_Y$. Since 
$$
K_{X}+B+\lambda_{1}H\sim_\Q 0/T
$$
 we have 
$$
K_{Y}+B_Y+\lambda_{1}H_Y\sim_\Q 0/S
$$ 
This combined with  $K_{Y}+B_Y\sim_\Q 0/S$
implies that $H_Y\sim_\Q 0/S$. In particular, since $K_{Y}+B_Y$ is semi-ample$/T$, 
the map $S\to T$ is birational otherwise $H_Y$ would be numerically negative on 
some curves not contained in $H_Y$ which is not possible as $H_Y\ge 0$.
Moreover, 
$$
K_{Y}+B_Y+\lambda H_Y\sim_\R 0/S
$$ 
Therefore, by Theorem \ref{t-main-klt-trivial}, $(Y/Z,B_Y+\lambda H_Y)$ 
has a log minimal model which is also a log minimal model of 
$(X/Z,B+\lambda H)$. Now the termination follows from Theorem \ref{t-mmodel-term-scaling}.
\end{proof}

An application of the above theorems concerns a certain LMMP with scaling which 
will be used in the proof of Lemmas \ref{l-s-termination} and \ref{l-acc-s-termination} and 
Proposition \ref{p-acc-main-vertical}.

\begin{lem}[cf. {[\ref{B-II}, Lemma 4.1]}]\label{l-LMMP-scaling}
Let $(X/Z,B)$ be a $\Q$-factorial lc pair such that $f\colon X\to Z$ is 
surjective, $B$ is a $\Q$-divisor, and $(K_X+B)|_F\sim_\Q 0$ where $F$ is the generic fibre of $f$.
Moreover, assume that 

$\bullet$ every lc centre of $(X/Z,B)$ is contained in $\Supp \rddown{B}$,

$\bullet$ $K_X+B+cC$ is lc and nef$/Z$ for some $\Q$-divisor $C\ge 0$ and some 
rational number $c>0$,

$\bullet$ $K_X+B\sim_\Q P+C/Z$ for some $\Q$-divisor $P\ge 0$ with $\Supp P=\Supp \rddown{B}$.\\ 
Then, we can run an LMMP$/Z$ on $K_X+B$ with scaling of $cC$ so that  $\lambda:=\lim_{i\to \infty}\lambda_i=0$ 
where $\lambda_i$  are the numbers that appear in the LMMP. (Note that we are not claiming that 
the LMMP terminates). 
\end{lem}
\begin{proof}
\emph{Step 1.} 
Since 
$$
(P+C)|_F \sim_\Q(K_X+B)|_F\sim_\Q  0
$$ 
both $P$ and $C$ are 
vertical$/Z$, in particular, $\rddown{B}$ is also vertical$/Z$. 
On the other hand, 
for each rational number $t\in (0,1]$ we can write 
\begin{equation*}
\begin{split}
K_X+B+tcC\sim_\Q & (1+\epsilon)(K_X+B)-\epsilon P-\epsilon C+tcC \\
\sim_\Q & 
(1+\epsilon)(K_X+B-\frac{\epsilon}{1+\epsilon} P+\frac{tc-\epsilon }{1+\epsilon}C)/Z
\end{split}
\end{equation*}
for some sufficiently small rational number $\epsilon>0$. Under our assumptions, 
$$
(X,B-\frac{\epsilon}{1+\epsilon} P+\frac{tc-\epsilon}{1+\epsilon}C)
$$ 
is klt and 
$$
(K_X+B-\frac{\epsilon}{1+\epsilon} P+\frac{tc-\epsilon}{1+\epsilon}C)|_F=(K_X+B)|_F\sim_\Q 0
$$ 
 In particular, if $K_X+B+tcC$ is nef$/Z$, then 
 $$
 K_X+B-\frac{\epsilon}{1+\epsilon} P+\frac{tc-\epsilon}{1+\epsilon}C
 $$ 
 is nef$/Z$ hence semi-ample$/Z$ by Theorem \ref{t-main-klt}; this in turn implies that  
$K_X+B+tcC$ is semi-ample$/Z$. We will use this observation 
on $X$ and on the birational models that will be constructed.

\emph{Step 2.} Put $Y_1:=X$, $B_1:=B$, and $C_1:=C$. Let $\lambda_1\ge 0$ be the 
smallest number such that $K_{Y_1}+B_1+\lambda_1 cC_1$ is nef$/Z$. We may assume that $\lambda_1>0$. 
We will show that $\lambda_1$ is rational. 
Pick a rational number $\lambda'\in (0,\lambda_1)$. In view of Step 1, there is a boundary $\Delta$ 
and a rational number $\epsilon>0$ such that 
$(Y_1/Z,\Delta)$ is klt and 
$$
 K_{Y_1}+B_1+\lambda' cC_1 \sim_\Q (1+\epsilon) (K_{Y_1}+\Delta) /Z
$$
Now $s=\frac{\lambda_1-\lambda'}{1+\epsilon}$ is the smallest number such that $K_{Y_1}+\Delta+scC_1$ 
is nef$/Z$. By [\ref{B}, Lemma 3.1], $s$ is rational which means that 
$\lambda_1$ is also rational (note that we did not apply [\ref{B}, Lemma 3.1] directly to 
$(Y_1/Z,B_1)$ because this pair may not be dlt).

By Step 1, $K_{Y_1}+B_1+\lambda_1cC_1$ is semi-ample$/Z$; let $Y_1\to V_1/Z$ be the 
associated contraction. Run the LMMP$/V_1$ on $K_{Y_1}+B_1+tcC_1$ with scaling of 
an ample$/V_1$ divisor, for some rational number $t\in (0,\lambda_1)$. This terminates 
with a good log minimal model of $(Y_1/V_1,B_1+tcC_1)$ by Step 1 and by Theorem 
\ref{t-main-klt}. The LMMP is also an LMMP$/V_1$ on $K_{Y_1}+B_1$ with scaling of $\lambda_1 cC_1$ 
because 
$$
K_{Y_1}+B_1+tcC_1\sim_\Q -(\lambda_1-t)cC_1\sim_\Q \frac{(\lambda_1-t)}{\lambda_1}(K_{Y_1}+B_1)/V_1
$$ 
So, we get a model $Y_2$ such that $K_{Y_2}+B_2$ is semi-ample$/V_1$ 
 where $B_2$ is the birational transform of $B_1$ 
(similar notation will be used for other divisors and models). Now 
since $K_{Y_2}+B_2+\lambda_1cC_2$ is the pullback of some ample$/Z$ divisor on $V_1$, 
$$
K_{Y_2}+B_2+\lambda_1cC_2+\delta (K_{Y_2}+B_2)
$$ 
is semi-ample$/Z$ for some sufficiently small $\delta>0$. 
Put it in another way, $K_{Y_2}+B_2+\tau cC_2$ is semi-ample$/Z$ for some rational number 
$\tau<\lambda_1$. We can consider $Y_1\bir Y_2$ as a partial 
LMMP$/Z$ on $K_{Y_1}+B_1$ with scaling of $cC_1$.
We can continue as before. That is, let $\lambda_2\ge 0$ be the smallest number such that 
$K_{Y_2}+B_2+\lambda_2 cC_2$ is nef$/Z$, and so on (note that $\lambda_1>\tau\ge \lambda_2$).
 This process is an LMMP$/Z$ on $K_X+B$ with scaling of $cC$. 
 The numbers $\lambda_i$ that appear in the LMMP satisfy $\lambda:=\lim_{i\to \infty}\lambda_i\neq \lambda_j$ 
for any $j$. 

\emph{Step 3.} In Step 2, we constructed an LMMP$/Z$ 
on $K_X+B$ with scaling of $cC$.
After modifying the notation, 
we may assume that the LMMP consists of only a 
sequence $X_i\bir X_{i+1}/Z_i$ of log flips, $X_1=X$, and that 
$K_{X_i}+B_i+\lambda_i cC_i$ is nef$/Z$ but numerically trivial$/Z_i$.

If $\lambda>0$, then the LMMP is also an LMMP$/Z$ on 
$K_X+B+\lambda' cC$ with scaling of $(1-\lambda')cC$ for some 
rational number $\lambda'\in (0,\lambda)$.
As in Step 1, there is a rational number $\epsilon>0$ and a klt $(X/Z,\Delta)$ 
such that 
$$
K_X+B+\lambda' cC\sim_\Q (1+\epsilon)(K_X+\Delta)/Z
$$
So, we can consider the LMMP as an LMMP$/Z$ on $K_X+\Delta$ with scaling of 
$\frac{1-\lambda'}{1+\epsilon}cC$. Since $(X/Z,\Delta+\frac{1-\lambda'}{1+\epsilon}cC)$ 
is klt and $C$ is vertical$/Z$, 
Theorem \ref{t-main-klt-termination} (2) implies that the LMMP terminates. 
So,  $\lambda=0$.
\end{proof}

%%%%%%%%%%%%%%%%%%%%%%%%%%%%%%%%%%%%%
%%%%%%%%%%%%%%%%%%%%%%%%%%%%%%%%%%%%%
%%%%%%%%%%%%%%%%%%%%%%%%%%%%%%%%%%%%%
\vspace{0.3cm}
\section{\textbf{Proof of Theorem \ref{conj-main-general} and Corollary \ref{t-lc-flips}}}

Corollary \ref{t-lc-flips} follows from Theorem \ref{conj-main-general} easily (see the 
end of this section).
We divide the proof of Theorem \ref{conj-main-general} into two cases. One case is when 
every lc centre of $(X/Z,B)$ is vertical$/Z$ (the vertical case) and the other case is 
when some lc centre is horizontal$/Z$ (the horizontal case). Each case needs a different 
kind of argument. Here we give a brief account of the main ideas (a very similar line of thought is used to prove Theorem \ref{t-acc-main-general-modified}). 

Suppose that $(X/Z,B+A)$ is as in the statement of Theorem \ref{conj-main-general}. 
By taking a $\Q$-factorial dlt blowup we can assume that $(X/Z,B)$ is $\Q$-factorial dlt.
First note that if $A$ intersects the generic fibre of $X\to Z$, then $K_X+B$ is not 
pseudo-effective$/Z$ and the result follows from [\ref{BCHM}].  
So, we could assume that $A$ is vertical$/Z$. 

Now assume that every lc centre of $(X/Z,B)$ is vertical$/Z$ which   
means that $\rddown{B}$ is vertical$/Z$. By Theorem \ref{t-mmodel-term-scaling},
 the termination statement in Theorem \ref{conj-main-general} (3) holds for $(X/Z,B)$ 
 if statement (1) holds. Thus, 
we only need to construct a good log minimal model of $(X/Z,B)$.
Pick a sequence 
$t_1>t_2>\dots$ of sufficiently small positive rational numbers such that $\lim_{i\to \infty} t_i=0$. 
Then, $(X/Z,B-t_i \rddown{B})$ 
is a klt pair for each $i$. So, by Theorem \ref{t-main-klt}, for each $i$,
we get a good log minimal model $(Y_i/Z,B_{Y_i}-t_i \rddown{{B}_{Y_i}})$ of 
$(X/Z,B-t_i \rddown{B})$ by running some LMMP$/Z$ on $K_X+B-t_i \rddown{B}$. 

Assume that $Y_i=Y_{i+1}$ for every $i\gg 0$ and let $Y$ be this common model. Then,   
a simple calculation on log discrepancies show that $(Y/Z,B_{Y})$ is 
actually a weak lc model of $(X/Z,B)$ from which we get a log minimal model 
as in Corollary \ref{c-wlc-to-mmodel}.
In general the $Y_i$ may be different 
but at least we could assume that they are all isomorphic in codimension one. 
In particular, if $\rddown{B_{Y_1}}=0$, then we could replace each $Y_i$ 
with $Y_1$ and proceed as before. We may then assume that $\rddown{B_{Y_1}}\neq 0$.

Since $K_{Y_1}+B_{Y_1}-t_1\rddown{B_{Y_1}}$ is semi-ampleness$/Z$,  
$$
K_{Y_1}+B_{Y_1}-t_1\rddown{B_{Y_1}}\sim_\Q C_{Y_1}/Z
$$
for some $C_{Y_1}\ge 0$ so that $K_{Y_1}+B_{Y_1}+C_{Y_1}$ is lc. 
If $t_i'=\frac{t_i}{t_1-t_i}$ and $i\neq 1$, then 
$$
K_{Y_1}+B_{Y_1}+t_i'C_{Y_1}\sim_\Q (1+t_i')(K_{Y_1}+B_{Y_1}-t_i\rddown{B_{Y_1}})/Z
$$
We can make sure that $t_i'\le 1$ for every $i>1$.
Next we run an LMMP$/Z$ on $K_{Y_2}+B_{Y_2}$ with scaling of $t_2'C_{Y_2}$. 
We use special termination (see Lemma \ref{l-s-termination}) to show that the LMMP terminates with a model $Y$ 
. We can proceed as before by replacing each $Y_i$ with $Y$, for $i\gg 0$ (see Proposition \ref{p-main-vertical} 
for more details). 
Finally, we will use Theorem 
\ref{conj-s-ampleness} to show that the log minimal model we have constructed 
is actually good (see Proposition \ref{p-main-vertical-sa}).

Now the horizontal case: assume that some lc centre of $(X/Z,B)$ is horizontal$/Z$, that is, some component of $\rddown{B}$ is horizontal$/Z$. The above arguments 
do not work since $K_X+B-t_i \rddown{B}$ is not pseudo-effective$/Z$ so we do not 
have a log minimal model of  $(X/Z,B-t_i \rddown{B})$. However, $(X/Z,B-t_i \rddown{B})$ 
has a Mori fibre space. Assume that for some $i$ we already have a Mori fibre structure on 
$X$, that is, we have a $K_X+B-t_i \rddown{B}$-negative extremal contraction 
$g\colon X\to T/Z$ which is of fibre type, that is, $\dim X>\dim T$. The condition 
$K_X+B+A\sim_\Q 0/Z$ and the pseudo-effectivity of $K_X+B$ ensures that 
$K_X+B\sim_\Q 0/T$, in particular, $K_X+B\sim_\Q g^*M/Z$ for some 
$\Q$-Cartier divisor $M$ on $T$. Now there is a component of $\rddown{B}$, say $S$, which maps 
onto $T$ since $\rddown{B}$ is ample$/T$. By applying induction to 
$$
K_S+B_S+A_S:=(K_X+B+A)|_S
$$ 
we deduce that $(S/Z,B_S)$ has a good log minimal model hence 
the algebra $\mathcal{R}(S/Z, K_S+B_S)$ 
is finitely generated over $\x{O}_Z$. 
It turns out that  this implies that  $\mathcal{R}(T/Z, M)$  is finitely generated (see Lemma \ref{t-fg-pullbacks})
hence $\mathcal{R}(X/Z, K_X+B)$ is also finitely generated. 
Next, we can apply Theorem \ref{t-main-under-fg} (see Lemma \ref{l-horizontal-trivial-1}). 
In general, there may not be any Mori fibre structure on $X$ but there is such a structure 
on some birational model of $X$ which can be used in a somewhat similar way 
(see Proposition \ref{t-main-horizontal} and Lemma \ref{l-horizontal-trivial-2}).

\subsection*{The vertical case}
In this subsection, we deal with Theorem \ref{conj-main-general} in the vertical case, i.e. 
when every lc centre of $(X/Z,B)$ is vertical$/Z$, 
in particular, when $X\to Z$ is birational. First we prove a kind of special termination  
which will enable us to do induction on dimension.
It is helpful to recall Remarks  \ref{rem-lifting-flips} and \ref{rem-LMMP-scaling-induction} before reading the 
proof of the next result.

\begin{lem}\label{l-s-termination}
Let $(X/Z,B+A)$ be a lc pair of dimension $d$ as in Theorem \ref{conj-main-general} 
such that $(X/Z,B)$ satisfies the assumptions of Lemma \ref{l-LMMP-scaling}.
Then, assuming Theorem \ref{conj-main-general} in dimension $d-1$, there is an LMMP$/Z$ 
on $K_X+B$ with scaling of $cC$ which terminates.
\end{lem}
\begin{proof}
By Lemma \ref{l-LMMP-scaling}, there is an LMMP$/Z$ 
on $K_X+B$ with scaling of $cC$ such that $\lambda=\lim_{i\to \infty} \lambda_i=0$. We can assume that 
the LMMP consists of only log flips. Let $(X'/Z,B')$ be a $\Q$-factorial dlt blowup 
of $(X/Z,B)$ and let $C'$ be the birational transform of $C$. 
By Remark \ref{rem-lifting-flips} (1)(2), 
we can lift the above LMMP$/Z$ on $K_X+B$ to an LMMP$/Z$ on $K_{X'}+B'$ with scaling of $cC'$. 
We could assume that the latter LMMP consists of only log flips. 

Let $S$ be a component of 
$\rddown{B'}$ and let $T$ be the normalisation of the image 
of $S$ in $Z$. Put 
 $K_{S}+B_{S}:=(K_{X'}+B')|_{S}$ and let $(S',B_{S'})$ be a $\Q$-factorial dlt blowup 
of $(S,B_{S})$. 
By construction, 
$$
K_{S'}+B_{S'}+A_{S'}\sim_\Q 0/T
$$ 
where $A_{S'}$ is the pullback of $A$. 
By Remark \ref{rem-LMMP-scaling-induction}, we may assume that the 
LMMP on $K_{X'}+B'$ induces an LMMP$/T$ on $K_{S'}+B_{S'}$
with scaling of $cC_{S'}$ where  $C_{S'}$ is the pullback of $C$. 

Since $\lambda=\lim_{i\to \infty} \lambda_i=0$, $K_{S'}+B_{S'}$ is pseudo-effective$/T$ which means that 
$A_{S'}$ is vertical$/T$. By induction, $(S'/T,B_{S'})$ has a log minimal model.
Thus, by Theorem \ref{t-mmodel-term-scaling}, the above LMMP$/T$ on $K_{S'}+B_{S'}$ 
terminates. Thus,
the LMMP on $K_{X'}+B'$  terminates near $S$. The same argument applied to each 
component of $\rddown{B'}$ shows that the LMMP on $K_{X'}+B'$  terminates near $\rddown{B'}$. 

By assumptions, $K_X+B\sim_\Q P+C/Z$ where $P\ge 0$ and $\Supp P=\Supp \rddown{B}$. 
Moreover,  $\Supp \rddown{B}$ contains all the lc centres of $(X/Z,B)$. Thus, 
there is a $\Q$-divisor $P'\ge 0$ with $\Supp P'=\Supp \rddown{B'}$ such that 
 $K_{X'}+B'\sim_\Q P'+C'/Z$. So, each extremal ray contracted by the LMMP 
 on $K_{X'}+B'$ intersects $P'$.  
But the LMMP terminates near $\rddown{B'}$ which in turn implies that it terminates near $P'$ therefore it terminates everywhere.
\end{proof}

Next we use the above special termination to derive parts $(1)$ and $(3)$ of 
Theorem \ref{conj-main-general}, in the vertical case, from Theorem \ref{conj-main-general} 
in lower dimensions. The semi-ampleness statement in part $(2)$ 
will be proved afterwards using Theorem \ref{conj-s-ampleness}.

\begin{prop}\label{p-main-vertical}
Theorem \ref{conj-main-general} in dimension $d-1$ implies Theorem \ref{conj-main-general} $(1)(3)$
 in dimension $d$ in the vertical case, i.e. when every lc centre of $(X/Z,B)$  is vertical$/Z$.
\end{prop}
\begin{proof}
\emph{Step 1.} After taking a $\Q$-factorial dlt blowup using Corollary 
\ref{c-Q-factorial-blup} we may assume that 
$(X/Z,B)$ is $\Q$-factorial dlt, in particular, $\rddown{B}$ is 
vertical$/Z$. Run an LMMP$/Z$ on $K_X+B$ 
with scaling of an ample$/Z$ divisor. If $K_X+B$ is not pseudo-effective$/Z$, 
then the LMMP ends up with a Mori fibre space by [\ref{BCHM}].
 So, from now on we assume that 
$K_X+B$ is pseudo-effective$/Z$. 
 By Theorem \ref{t-mmodel-term-scaling}, the LMMP terminates 
if we show that $(X/Z,B)$ has a log minimal model. 
Since $K_X+B\sim_\Q -A/Z$, $A|_F=0$ where $F$ 
is the generic fibre of $f\colon X\to Z$. Thus, $A$ is vertical$/Z$ and 
$(K_X+B)|_F\sim_\Q 0$.

\emph{Step 2.} Let $t_1>t_2>\cdots$ be a sequence of sufficiently small rational numbers with 
$\lim_{i\to \infty} t_i=0$. Each $(X/Z,B-t_i\rddown{B})$ is klt and 
$(K_X+B-t_i\rddown{B})|_F\sim_\Q 0$ hence 
by Theorem \ref{t-main-klt} each $(X/Z,B-t_i\rddown{B})$ has a good log 
minimal model $(Y_i/Z,B_{Y_i}-t_i\rddown{B_{Y_i}})$ so that 
$Y_i\bir X$ does not contract divisors. Moreover, $\Supp \rddown{B_{Y_i}}$ contains 
all the lc centres of $(Y_i/Z,B_{Y_i})$ because $(Y_i/Z,B_{Y_i}-t_i\rddown{B_{Y_i}})$ is klt 
which means that $(Y_i/Z,B_{Y_i})$ is klt outside 
$\Supp \rddown{B_{Y_i}}$.
Now, since $A$ and $\rddown{B}$ are vertical$/Z$, there are vertical$/Z$ $\Q$-divisors $M,N\ge 0$ such that 
$$
K_X+B\sim_\Q -A\sim_\Q M/Z ~~~\mbox{and}~~~ -\rddown{B}\sim_\Q N/Z
$$ 
and
$$
K_X+B-t_i\rddown{B}\sim_\Q M+t_iN/Z
$$ 
So, any prime divisor contracted 
by $X\bir Y_i$ is a component of $M+N$ hence after replacing the sequence with a subsequence 
we can assume that the maps $X\bir Y_i$ contract the same divisors , i.e. $Y_i$ are 
isomorphic in codimension one.

\emph{Step 3.} Assume that $Y_i=Y_{i+1}$ for $i\gg 0$, and let $Y$ be this common model.
Since $K_Y+B_Y-t_i\rddown{B_{Y}}$ is nef$/Z$ for each $i\gg 0$, $K_Y+B_Y$ is also nef$/Z$. 
For any prime divisor $D$ on $X$ and each $i\gg 0$, we have 
$$
a(D,X,B)\le a(D,X,B-t_i\rddown{B})\le a(D,Y,B_Y-t_i\rddown{B_Y}) 
$$
which implies that 
$$
a(D,X,B)\le \lim_{i\to \infty} a(D,Y,B_Y-t_i\rddown{B_Y})=a(D,Y,B_Y)
$$
Therefore, $(Y/Z,B_Y)$ is a weak lc model of $(X/Z,B)$, and  
by Corollary \ref{c-wlc-to-mmodel}, we can construct a log minimal model of $(X/Z,B)$ 
as required.

\emph{Step 4.} 
Assume that $\rddown{B_{Y_1}}=0$. Then,  $K_{Y_1}+B_{Y_1}-t_i\rddown{B_{Y_1}}$ is nef$/Z$ 
for each $i$, and since $Y_1\bir Y_i$ is an isomorphism in codimension one,  
we can replace each $Y_i$ with $Y_1$ and then apply Step 3. 

 From now on we assume that $\rddown{B_{Y_1}}\neq 0$. 
 Since $(Y_1/Z,B_{Y_1}-t_1\rddown{B_{Y_1}})$ 
is klt and $K_{Y_1}+B_{Y_1}-t_1\rddown{B_{Y_1}}$ is nef$/Z$, 
$K_{Y_1}+B_{Y_1}-t_1\rddown{B_{Y_1}}$ is semi-ample$/Z$ by Theorem \ref{t-main-klt}.
 Thus,
$$
K_{Y_1}+B_{Y_1}-t_1\rddown{B_{Y_1}}\sim_\Q C_{Y_1}/Z
$$
for some $C_{Y_1}\ge 0$ such that 
$K_{Y_1}+B_{Y_1}+C_{Y_1}$ is lc, in particular, $\Supp C_{Y_1}$ does not contain any lc centre 
of $(Y_1/Z,B_{Y_1})$. For each 
rational number $t$ we have 
$$
K_{Y_1}+B_{Y_1}+tC_{Y_1}\sim_\Q (1+t)(K_{Y_1}+B_{Y_1}-\frac{tt_1}{1+t}\rddown{B_{Y_1}})/Z
$$
In particular, if $t_i'=\frac{t_i}{t_1-t_i}$ and $i\neq 1$, then 
$$
K_{Y_1}+B_{Y_1}+t_i'C_{Y_1}\sim_\Q (1+t_i')(K_{Y_1}+B_{Y_1}-t_i\rddown{B_{Y_1}})/Z
$$
We may assume that $2t_i\le t_1$ for each $i>1$ which implies that $t_i'\in [0,1]$. So, $(Y_i/Z,B_{Y_i}+t_i'{C_{Y_i}})$ is a weak lc model of $(Y_1/Z,B_{Y_1}+t_i'{C_{Y_1}})$ for each $i>1$.
Note that $(Y_1/Z,B_{Y_1}+t_i'{C_{Y_1}})$ is lc because $t_i'\le 1$.

\emph{Step 5.} By construction,
$$
K_{Y_2}+B_{Y_2}+A_{Y_2}\sim_\Q 0/Z
$$
and every lc centre of 
$(Y_2/Z,B_{Y_2})$ is contained in $\rddown{B_{Y_2}}$ because $(Y_2/Z,B_{Y_2}-t_2\rddown{B_{Y_2}})$ 
is klt. On the other hand, 
since  $A$ is vertical$/Z$, $A_{Y_2}$ is also vertical$/Z$, and if we put $P_{Y_2}=t_1\rddown{B_{Y_2}}$, 
then by construction 
$$
K_{Y_2}+B_{Y_2}\sim_\Q {P_{Y_2}}+ C_{Y_2}/Z
$$
Now $K_{Y_2}+B_{Y_2}+{t_2'}C_{Y_2}$ is nef$/Z$, so by Lemma \ref{l-s-termination}, we can run an 
LMMP$/Z$ on $K_{Y_2}+B_{Y_2}$ with scaling of $t_2'C_{Y_2}$ 
which terminates on a model $Y$ on which $K_Y+B_Y+\delta C_Y$ 
is nef$/Z$ for any sufficiently small $\delta\ge 0$.  
Since $Y_2\bir Y_i$ is an isomorphism in codimension one,   
$K_{Y_i}+B_{Y_i}+t_i'C_{Y_i}$ is nef$/Z$, and $\lim_{i\to \infty} t_i'=0$, we deduce that 
$K_{Y_2}+B_{Y_2}$ is (numerically) a limit of movable$/Z$ $\R$-divisors. 
So, the LMMP does not contract any divisors, i.e. $Y\bir Y_i$ is an isomorphism in codimension one.

For any $i\gg 0$, 
$K_Y+B_{Y}+t_i'{C_{Y}}$ is nef$/Z$ which in turn implies that 
$K_Y+B_{Y}-t_i\rddown{B_Y}$ is also nef$/Z$.
Since $Y\bir Y_i$ is an isomorphism in codimension one and 
since  $(Y_i/Z,B_{Y_i}-t_i\rddown{B_{Y_i}})$ is a log minimal model of 
 $(X/Z,B-t_i\rddown{B})$,  $(Y/Z,B_Y-t_i\rddown{B_Y})$ is a log minimal model of 
$(X/Z,B-t_i\rddown{B})$, for every $i\gg 0$.
By Step 3, we are done.
\end{proof}

\begin{prop}\label{p-main-vertical-sa}
Theorem \ref{conj-main-general} in dimension $d-1$ implies Theorem \ref{conj-main-general} in dimension $d$
in the vertical case, i.e. when every lc centre of $(X/Z,B)$  is vertical$/Z$.
\end{prop}
\begin{proof}
After taking a $\Q$-factorial dlt blowup using Corollary 
\ref{c-Q-factorial-blup} we may assume that 
$(X/Z,B)$ is $\Q$-factorial dlt, in particular, $\rddown{B}$ is 
vertical$/Z$. By Proposition \ref{p-main-vertical}, any LMMP$/Z$ on 
$K_X+B$ with scaling of an ample$/Z$ divisor ends up with a Mori fibre space or 
a log minimal model $(Y/Z,B_Y)$. 
 Assume that  $(Y/Z,B_Y)$ is a log minimal model.
Let $g\colon W\to X$ and $h\colon W\to Y$ be a common resolution, and let 
$A_Y:=h_*g^*A$. Then, 
$$
K_Y+B_Y+A_Y=h_*g^*(K_X+B+A)\sim_\Q 0/Z
$$
where we use the fact that $h_*g^*(K_X+B)=K_Y+B_Y$ which in turn follows from the fact that 
$g^*(K_X+B)-h^*(K_Y+B_Y)$ is exceptional$/Y$, by Remark \ref{rem-on-m-models}.
Also, $(Y/Z,B_Y+A_Y)$ is lc because 
$$
 h^*(K_Y+B_Y+A_Y)=g^*(K_X+B+A)
$$
 So, by replacing $(X/Z,B+A)$ with $(Y/Z,B_Y+A_Y)$ we may assume that $K_X+B$  is nef$/Z$. 
 It remains to prove that $K_X+B$ is semi-ample$/Z$.

Run an LMMP$/Z$ on $K_X+B-\epsilon \rddown{B}$ with scaling of some ample$/Z$ divisor, for some sufficiently 
small rational number $\epsilon>0$. Since $K_X+B$ is nef$/Z$ and $K_X+B+A\sim_\Q 0/Z$, 
$(K_X+B)|_F\sim_\Q 0$ where $F$ is the generic fibre of $X\to Z$. Moreover, since $\rddown{B}$ is 
vertical$/Z$, 
$$
(K_X+B-\epsilon \rddown{B})|_F\sim_\Q 0
$$ So, by 
Theorem \ref{t-main-klt}, the LMMP terminates 
on a model $X'$ on which $K_{X'}+B_{X'}-\epsilon \rddown{B_{X'}}$ is semi-ample$/Z$. 
On the other hand, by [\ref{B-II}, Proposition 3.2], $K_X+B$ is numerically trivial on each extremal ray 
contracted by the LMMP. Therefore, $K_{X'}+B_{X'}$ is also nef$/Z$. Another application of 
Theorem \ref{t-main-klt} shows that $K_{X'}+B_{X'}-\delta \rddown{B_{X'}}$ is semi-ample$/Z$ 
for any $\delta\in (0,\epsilon]$.
In addition, $X\bir X'$ is an isomorphism over the 
generic point of $Z$.

The pair $(X'/Z,B_{X'})$ is $\Q$-factorial and lc and  $(X'/Z,B_{X'}-\epsilon \rddown{B_{X'}})$ 
is klt. If $\rddown{B_{X'}}=0$ we are done so we can assume that $\rddown{B_{X'}}\neq 0$.
Let $(Y/Z,B_Y)$ be a $\Q$-factorial dlt blowup of $(X'/Z,B_{X'})$ and 
let $S$ be a component of $T:=\rddown{B_{Y}}$. Let $A_Y$ be the pullback of 
$A_{X'}$. Then, from 
$$
K_{X'}+B_{X'}+A_{X'}\sim_\Q 0/Z
$$ 
we get 
$$
K_{Y}+B_{Y}+A_{Y}\sim_\Q 0/Z
$$ 
which implies that 
$$
K_{S}+B_{S}+A_{S}\sim_\Q 0/Z
$$ 
where $K_S+B_S:=(K_Y+B_Y)|_S$ and $A_S:=A_Y|_S$.  
Now, by induction, $K_S+B_S$ is semi-ample$/Z$. On the other hand, if 
$P:= e^*\rddown{B_{X'}}$ and if $\delta>0$ is any sufficiently 
small rational number, where $e$ is the morphism $Y\to X'$, then 
$$
K_Y+B_Y-\delta P=e^*(K_{X'}+B_{X'}-\delta \rddown{B_{X'}})
$$ 
is semi-ample$/Z$. Now, by Theorem \ref{conj-s-ampleness}, 
$K_Y+B_Y$ is semi-ample$/Z$ hence $K_X+B$ is also semi-ample$/Z$.
\end{proof}

%%%%%%%%%%%%%%%%%%%%%%%%%%%%%%%%%%%%%

\subsection*{The horizontal case}
In this subsection, we deal with the horizontal case of Theorem \ref{conj-main-general}, 
that is, when some lc centre of $(X/Z,B)$ is horizontal$/Z$. 
First, we need the following result on finite generation of algebras.

\begin{lem}\label{t-fg-pullbacks}
Let $f\colon X\to Y/Z$ be a surjective morphism of normal varieties, projective 
over an affine variety $Z=\Spec R$, and $L$ a Cartier divisor on $Y$. 
If $R(X/Z,f^*L)$ is a finitely generated $R$-algebra,  
then $R(Y/Z,L)$ is also a finitely generated $R$-algebra.
\end{lem}

The lemma can be easily derived from the deep fact that $\x{O}_Y$ splits $f_*\x{O}_X$. 
However, a simpler proof (suggested by the referee and independently by Koll\'ar and Totaro) 
uses only basic commutative algebra: $R(X/Z,f^*L)$ 
is integral over $R(Y/Z,L)$ so the latter is finitely generated since the 
former is assumed to be finitely generated  (cf. [\ref{AM}, Proposition 7.8]). 
Okawa [\ref{Okawa}, Theorem 4.1] proves a more general statement.

We will reduce the horizontal case to the next lemma by finding a suitable 
Mori fibre space.

\begin{lem}\label{l-horizontal-trivial-1}
Assume Theorem \ref{conj-main-general} in dimension $d-1$.
Let $(X/Z,B+A)$ be of dimension $d$ as in Theorem \ref{conj-main-general}
such that $(K_X+B)|_F\sim_\Q 0$ for the generic fibre $F$ of $f\colon X\to Z$.
Moreover, assume that there is a contraction $g\colon X\to T/Z$ such that 

$(1)$ $K_X+B\sim_\Q 0/T$,

$(2)$ some lc centre of $(X/Z,B)$ is horizontal over $T$.\\
Then, $(X/Z,B)$ has a good log minimal model.
\end{lem}
\begin{proof}
By replacing $(X/Z,B)$ with a $\Q$-factorial dlt blowup, 
we can assume that $(X/Z,B)$ is $\Q$-factorial dlt and that there is a 
component $S$ of $\rddown{B}$ which is horizontal$/T$. Run an LMMP$/Z$ on 
$K_X+B$ with scaling of some ample$/Z$ divisor. Since termination and semi-ampleness$/Z$ are 
local on $Z$, 
we can assume that $Z$ is affine, say $\Spec R$. 
By Theorem \ref{t-mmodel-term-scaling}, 
the LMMP terminates with a 
good log minimal model if we prove that $(X/Z,B)$ has a good log minimal model. 

By adjunction 
define $K_S+B_S:=(K_X+B)|_S$ and let $A_S:=A|_S$.
Then, 
$$
K_S+B_S+A_S=(K_X+B+A)|_S\sim_\Q 0/Z
$$ 
and $K_S+B_S\sim_\Q 0/T$. Moreover, $(K_S+B_S)|_H\sim_\Q 0$ where 
$H$ is the generic fibre of the induced morphism $h\colon S\to Z$.

Since we are assuming Theorem \ref{conj-main-general} in dimension $d-1$, 
$(S/Z,B_S)$ has a good log minimal model. 
Therefore, if $I(K_S+B_S)$ is Cartier for some $I\in\N$, 
then  $R(S/Z,I(K_S+B_S))$ is a finitely generated $R$-algebra (cf. [\ref{B-dam}]).
We can choose $I$ such that $I(K_X+B)$ is Cartier and such that 
$I(K_X+B)\sim g^*L$ for some Cartier divisor $L$ on $T$. So, 
$I(K_S+B_S)\sim e^*L$ where $e\colon S\to T$ is the induced morphism.
Now, by Lemma \ref{t-fg-pullbacks}, $R(T/Z,L)$ is a finitely 
generated $R$-algebra which in turn implies that $R(X/Z,I(K_X+B))$
is a finitely generated $R$-algebra since $X\to T$ is a contraction. 
Therefore, $R(X/Z,K_X+B)$
is a finitely generated $R$-algebra and according to Theorem \ref{t-main-under-fg}, $(X/Z,B)$ has a good log minimal model.
\end{proof}

\begin{prop}\label{t-main-horizontal}
Assume Theorem \ref{conj-main-general} in dimension $d-1$, and
assume Theorem \ref{conj-main-general} in dimension $d$ in the vertical case.
Then, Theorem \ref{conj-main-general} holds in dimension $d$ in the horizontal case, i.e. 
when some lc centre of $(X/Z,B)$ is horizontal over $Z$.
\end{prop}
\begin{proof}
We can assume that $(X/Z,B)$ is $\Q$-factorial dlt and that $f\colon X\to Z$ is a contraction. 
By assumptions, some component of $\rddown{B}$ is 
horizontal$/Z$.
Run an LMMP$/Z$ on $K_X+B$ with scaling of some ample$/Z$ divisor. If $A$ is not vertical$/Z$, 
then $K_X+B$ is not pseudo-effective$/Z$ hence the LMMP terminates with a 
Mori fibre space by [\ref{BCHM}]. So, we can assume that $A$ is vertical$/Z$ hence 
$(K_X+B)|_F\sim_\Q 0$ where $F$ is the generic fibre of $f\colon X\to Z$. 
By Theorem \ref{t-mmodel-term-scaling}, the LMMP terminates with a good log minimal model 
if we show that $(X/Z,B)$ has a 
good log minimal model.

If $\epsilon>0$ is a sufficiently small rational number, then 
$(X/Z,B-\epsilon \rddown{B})$ is klt and $K_X+B-\epsilon \rddown{B}$ is not pseudo-effective$/Z$.
Thus, by [\ref{BCHM}] there is a 
Mori fibre space $(Y/Z,B_Y-\epsilon \rddown{B_Y})$ for $(X/Z,B-\epsilon \rddown{B})$ 
obtained by running an LMMP$/Z$ on $K_X+B-\epsilon \rddown{B}$. 
Let $g\colon Y\to T/Z$ be the $K_Y+B_Y-\epsilon \rddown{B_Y}$-negative extremal 
contraction which defines the Mori fibre space structure, and let $R$ be the corresponding 
extremal ray. Since $A$ is vertical$/Z$, $A_Y$ is vertical over $T$. So, from  
$$
K_Y+B_Y+A_Y\sim_\Q 0/Z
$$
we deduce that $(K_Y+B_Y)\cdot R=0$ which in turn implies that $K_Y+B_Y\sim_\Q 0/T$. Since 
$K_Y+B_Y-\epsilon \rddown{B_Y}$ is numerically negative$/T$,  $\rddown{B_Y}\neq 0$.
We could apply Lemma \ref{l-horizontal-trivial-1} to get a good log 
minimal model of $(Y/Z,B_Y)$ but this does not give a good log minimal model of $(X/Z,B)$ 
because we cannot easily compare the singularities of $(Y/Z,B_Y)$ and $(X/Z,B)$.
We need to find a model which is closely related to both $(Y/Z,B_Y)$ and $(X/Z,B)$.

Let $(W/Z,B_W)$ be a log smooth model of $(X/Z,B)$ of type (1) as in Definition \ref{d-log-smooth-model} 
such that $W$ dominates $Y$, say by a morphism $h\colon W\to Y$. 
 We can write 
$$
K_W+B_W=h^*(K_Y+B_Y)+G
$$ 
where $G$ is exceptional$/Y$. Run the LMMP$/Y$ on $K_W+B_W$ with scaling of some ample$/Y$ divisor. 
By Theorem \ref{t-exc-LMMP}, at some step of this LMMP we reach a model $X'$ on which $G_{X'}\le 0$. Also, 
$$
K_{X'}+B_{X'}-G_{X'}\sim_\Q 0/T
$$ 
as it is the 
pullback of $K_Y+B_Y$. 
Since $K_{X}+B$ is pseudo-effective$/Z$, $K_W+B_W$ and $K_{X'}+B_{X'}$ are pseudo-effective over both 
$Z$ and $T$. Thus, $G_{X'}$ is 
vertical$/T$ hence $K_{X'}+B_{X'}\sim_\Q 0$ over the generic point of $T$. 
By construction, $\rddown{B_Y}$ is ample$/T$ hence in particular some of its components are horizontal$/T$. 
This implies that some component of $\rddown{B_{X'}}$ is horizontal$/T$ because $\rddown{B_Y}$ is the pushdown of $\rddown{B_{X'}}$.

Now by Lemma \ref{l-horizontal-trivial-2} below,  we can run an LMMP$/T$ on $K_{X'}+B_{X'}$ 
which ends up with a model $X''$ 
on which $K_{X''}+B_{X''}$ is semi-ample$/T$ (to be more precise, in place of $X,Y,Z,B,A$ in the lemma use  
$X'$, $Y$, $T$, $B_{X'}$, $-G_{X'}$ respectively). Since $K_{X'}+B_{X'}\sim_\Q 0$ over 
the generic point of $T$, $X'\bir X''$ is an isomorphism over the generic point of $T$.
So, if $X''\to T''/T$ is the contraction associated 
to $K_{X''}+B_{X''}$, then $T''\to T$ is birational
and some component of $\rddown{B_{X''}}$ is horizontal$/T''$. 

On the other hand,  if we denote $X'\to Y$ by $e$, then 
we can write 
$$
K_{X'}+B_{X'}+A_{X'}=e^*(K_Y+B_Y+A_Y)
$$ 
by taking $A_{X'}:=-G_{X'}+e^*A_Y$. Thus,  
$$
K_{X'}+B_{X'}+A_{X'}\sim_\Q 0/Z
$$ 
and 
$$
K_{X''}+B_{X''}+A_{X''}\sim_\Q 0/Z
$$ 
 where $A_{X''}$ is the birational transform of $A_{X'}$. 
 Since $K_{X''}+B_{X''}$ is pseudo-effective$/Z$, $A_{X''}$ is vertical$/Z$ and 
 $K_{X''}+B_{X''}\sim_\Q 0$ over the generic point of $Z$.
 Now we can apply Lemma \ref{l-horizontal-trivial-1} (by taking $X,Z,T,B,A$ in the Lemma 
 to be $X'',Z,T'',B_{X''},A_{X''}$ respectively) to get a good log minimal model of 
 $({X''}/Z,B_{X''})$ which would give a good log minimal model of $(W/Z,B_W)$ hence of 
 $(X/Z,B)$, by Remark \ref{rem-log-smooth-model}. 
\end{proof}

\begin{lem}\label{l-horizontal-trivial-2}
Assume Theorem \ref{conj-main-general} in dimension $d$ in the vertical case.
Let $(X/Z,B+A)$ be of dimension $d$ as in Theorem \ref{conj-main-general}
such that $(K_X+B)|_F\sim_\Q 0$ for the generic fibre $F$ of $f\colon X\to Z$.
Moreover, assume that

$\bullet$ $(X/Z,B)$ is $\Q$-factorial dlt,

$\bullet$ there exist contractions $e\colon X\to Y$ and $g\colon Y\to Z$,

$\bullet$ $e$ is birational, $A$ is exceptional$/Y$, and $Y$ is $\Q$-factorial, 

$\bullet$ $D:=K_Y+B_Y-\epsilon \rddown{B_Y}$ is klt and $g$ is a $D$-negative 
extremal contraction  where $\epsilon>0$ and $B_Y:=e_*B$.

Then, any LMMP$/Z$ on $K_X+B$ with scaling of an ample$/Z$ divisor  
terminates with a good log minimal model of $(X/Z,B)$.
\end{lem}
\begin{proof}
By Theorem \ref{t-mmodel-term-scaling}, it is enough to prove that $(X/Z,B)$ has a 
good log minimal model.
By adding a small multiple of $A$ to $B$ we can assume that $\Supp A\subseteq \Supp B$.
Since $e$ is birational, by Theorem \ref{conj-main-general} in dimension $d$ in the vertical case, 
we can assume that $K_X+B$ is actually 
semi-ample$/Y$. By replacing $X$ with the lc model of $(X/Y,B)$ we can 
assume that $K_X+B$ is ample$/Y$ (we may loose the $\Q$-factorial dlt property of $(X/Z,B)$ 
but we will recover it later).
  Since $-A$ is 
ample$/Y$, $\Supp A$ contains all the prime exceptional$/Y$ divisors on $X$. 
 In particular, 
$$
\Supp e^*\rddown{B_Y}\subseteq (\Supp \rddown{B}\cup \Supp A) \subseteq \Supp B
$$ 
By replacing $(X/Z,B)$ with a 
$\Q$-factorial dlt blowup, we can again assume that $(X/Z,B)$ is $\Q$-factorial dlt; 
note that this preserves the property $\Supp e^*\rddown{B_Y}\subseteq \Supp B$.

Since $A$ is exceptional$/Y$ and $K_X+B+A\sim_\Q 0/Z$, $K_Y+B_Y\sim_\Q 0/Z$. So, $\rddown{B_Y}$ is 
ample$/Z$ hence $e^*\rddown{B_Y}$ 
is semi-ample$/Z$. Thus, for a small rational number $\tau >0$ we can write
$$
K_X+B=K_X+B-\tau e^*\rddown{B_Y}+\tau e^*\rddown{B_Y}\sim_\Q K_X+\Delta/Z
$$ 
where $\Delta$ is some rational boundary such that 
$(X/Z,\Delta)$ is klt: since 
$(Y/Z,B_Y-\epsilon \rddown{B_Y})$ is klt, $\rddown{B_Y}$ contains all 
the lc centres of $(Y/Z,B_Y)$, in particular, the image of all the lc centres 
of $(X/Z,B)$; so, $\Supp e^*\rddown{B_Y}$ contains all the components of 
$\rddown{B}$ hence $(X/Z,B-\tau e^*\rddown{B_Y})$ is klt and we can indeed 
find $\Delta$ with the required properties.
By Theorem \ref{t-main-klt}, we can 
run an LMMP$/Z$ on  $K_X+\Delta$ which ends up with a good log minimal model 
of $(X/Z,\Delta)$ hence a good log minimal model of $(X/Z,B)$.
\end{proof}

\begin{proof}(of Theorem \ref{conj-main-general})
We argue by induction so in particular we may assume that Theorem \ref{conj-main-general} 
holds in dimension $d-1$. By Proposition \ref{p-main-vertical-sa}, 
Theorem \ref{conj-main-general} holds in dimension $d$ in the vertical case.
On the other hand, by Proposition \ref{t-main-horizontal}, 
 Theorem \ref{conj-main-general} also holds in dimension $d$ in the horizontal case.
\end{proof}

\begin{proof}(of Corollary \ref{t-lc-flips} )
First assume that $B$ has rational coefficients.
Since $-(K_X+B)$ is ample$/Y$, we can find 
$$
 A\sim_\Q -(K_X+B)/Y
$$
 such that $A\ge 0$, 
$(X/Z,B+A)$ is lc, and 
$$
K_X+B+A\sim_\Q 0/Y
$$ 
Now by Theorem  \ref{conj-main-general}, $(X/Y,B)$ has a good log minimal model 
and its lc model gives the $K_X+B$-flip. 
When $B$ is not rational, we can find 
rational lc divisors $K_X+B_i$ and real numbers $r_i>0$ such that 
$$
K_X+B=\sum r_i(K_X+B_i) ~~~~~~\mbox{and $\sum r_i=1$}
$$
 Moreover, we can assume that  
$-(K_X+B_i)$ is ample$/Y$ for every $i$. For each $i,j$ there is a rational number $b_{i,j}$ 
such that 
$$
K_X+B_i\equiv b_{i,j}(K_X+B_j)/Y
$$ 
By the cone theorem for lc pairs proved 
by Ambro [\ref{Ambro}] and Fujino [\ref{Fujino-lc-lmmp}], 
$$
K_X+B_i\sim_\Q b_{i,j}(K_X+B_j)/Y
$$
 Therefore, there is a morphism $X^+\to Y$
which gives the flip of $K_X+B_i$ for every $i$.
The morphism also gives the $K_X+B$-flip. 
\end{proof}

\vspace{0.3cm}
%%%%%%%%%%%%%%%%%%%%%%%%%%%%%%%%%%%%%
%%%%%%%%%%%%%%%%%%%%%%%%%%%%%%%%%%%%%
%%%%%%%%%%%%%%%%%%%%%%%%%%%%%%%%%%%%%

\section{\textbf{Proof of Theorem \ref{t-acc-main-general-modified}}}

In this section, we give the proof of Theorem \ref{t-acc-main-general-modified}
which is parallel to the arguments of section 6 with some small changes.

\subsection*{The vertical case}
First, we deal with Theorem \ref{t-acc-main-general-modified} in the vertical case.
It is helpful to recall Remarks  \ref{rem-lifting-flips} and \ref{rem-LMMP-scaling-induction} before reading the 
proof of the next special termination  result.

\begin{lem}\label{l-acc-s-termination}
Let $(X/Z,B)$ be a lc pair of dimension $d$ as in Theorem \ref{t-acc-main-general-modified}
such that it satisfies the assumptions of Lemma \ref{l-LMMP-scaling}.
Then, assuming Conjecture \ref{ACC} and Theorem \ref{t-acc-main-general-modified} in dimension $d-1$, 
there is an LMMP$/Z$ on $K_X+B$ with scaling of $cC$ which terminates.
\end{lem}
\begin{proof}
By Lemma \ref{l-LMMP-scaling}, there is an LMMP$/Z$ 
on $K_X+B$ with scaling of $cC$ such that $\lambda=\lim_{i\to \infty} \lambda_i=0$. We can assume that 
the LMMP consists of only log flips. Let $(X'/Z,B')$ be a $\Q$-factorial dlt blowup 
of $(X/Z,B)$ and let $C'$ be the birational transform of $C$. 
By Remark \ref{rem-lifting-flips} (1)(2), 
we can lift the above LMMP$/Z$ on $K_X+B$ to an LMMP$/Z$ on $K_{X'}+B'$ with scaling of $cC'$. 
We could assume that the latter LMMP consists of only log flips.

Let $S$ be a component of 
$\rddown{B'}$ and let $T$ be the normalisation of the image 
of $S$ in $Z$. Put 
 $K_{S}+B_{S}:=(K_{X'}+B')|_{S}$ and let $(S',B_{S'})$ be a $\Q$-factorial dlt blowup 
of $(S,B_{S})$. 
By Remark \ref{rem-LMMP-scaling-induction}, we may assume that the 
LMMP on $K_{X'}+B'$ induces an LMMP$/T$ on $K_{S'}+B_{S'}$
with scaling of $cC_{S'}$ where  $C_{S'}$ is the pullback of $C$. 
On the other hand, by assumptions, 
there is a non-empty open subset $U\subseteq Z$ such that 
$K_{X}+B\sim_\Q 0$ over $U$ and such that the generic point of each lc centre of 
$({X}/Z,B)$ is mapped into $U$. Let $V\subseteq T$ be the inverse image of $U$ 
under $T\to Z$. Then, $K_{S'}+B_{S'}\sim_\Q 0$ over $V$ and the generic point of each lc centre of 
 $(S'/T,B_{S'})$ is mapped into $V$.  

 By induction, $(S'/T,B_{S'})$ has a log minimal model. 
Since $\lambda=\lim_{i\to \infty} \lambda_i=0$, by Theorem \ref{t-mmodel-term-scaling}, the 
above LMMP$/T$ on $K_{S'}+B_{S'}$ terminates. Thus,
the LMMP$/Z$ on $K_{X'}+B'$  terminates near $S$. The same argument applied to each 
component of $\rddown{B'}$ shows that the LMMP$/Z$ on $K_{X'}+B'$  terminates near $\rddown{B'}$. 
By assumptions, $K_X+B\sim_\Q P+C/Z$ where $P\ge 0$ and $\Supp P=\Supp \rddown{B}$. 
Moreover,  $\Supp \rddown{B}$ contains all the lc centres of $(X/Z,B)$. Thus, 
there is a $\Q$-divisor $P'\ge 0$ with $\Supp P'=\Supp \rddown{B'}$ such that 
 $K_{X'}+B'\sim_\Q P'+C'/Z$.
 So, each extremal ray contracted by the LMMP$/Z$ 
 on $K_{X'}+B'$ intersects $P'$.  
But the LMMP terminates near $\rddown{B'}$ which in turn implies that it terminates near $P'$ therefore it terminates everywhere.
\end{proof}

\begin{prop}\label{p-acc-main-vertical}
Theorem \ref{t-acc-main-general-modified} in dimension $d-1$ implies Theorem 
\ref{t-acc-main-general-modified} (1)(3)
 in dimension $d$ in the vertical case, i.e. when every lc centre of $(X/Z,B)$  is vertical$/Z$.
\end{prop}
\begin{proof}
\emph{Step 1.} After taking a $\Q$-factorial dlt blowup using Corollary 
\ref{c-Q-factorial-blup} we may assume that 
$(X/Z,B)$ is $\Q$-factorial dlt, in particular, $\rddown{B}$ is 
vertical$/Z$. Run an LMMP$/Z$ on $K_X+B$ 
with scaling of an ample$/Z$ divisor.  By Theorem \ref{t-mmodel-term-scaling}, the LMMP terminates 
if we prove that $(X/Z,B)$ has a log minimal model. By assumptions, 
$K_X+B\sim_\Q 0$ over some non-empty open subset $U\subseteq Z$ and if $\eta$ is the 
generic point of any lc centre of $(X/Z,B)$, then $f(\eta)\in  U$. 

\emph{Step 2.} Let $t_1>t_2>\cdots$ be a sequence of sufficiently small rational numbers with 
$\lim_{i\to \infty} t_i=0$. Each $(X/Z,B-t_i\rddown{B})$ is klt and 
$(K_X+B-t_i\rddown{B})|_F\sim_\Q 0$ where $F$ is the generic fibre of $f$. 
Hence by Theorem \ref{t-main-klt} each $(X/Z,B-t_i\rddown{B})$ has a good log 
minimal model $(Y_i/Z,B_{Y_i}-t_i\rddown{B_{Y_i}})$ so that $Y_i\bir X$ does not 
contract divisors. Since we are assuming 
Conjecture \ref{ACC} in dimension $d$, we can assume that 
each $(Y_i/Z,B_{Y_i})$ is lc.
Moreover, $\Supp \rddown{B_{Y_i}}$ contains 
all the lc centres of $(Y_i/Z,B_{Y_i})$ because $(Y_i/Z,B_{Y_i}-t_i\rddown{B_{Y_i}})$ is klt 
which means that $(Y_i/Z,B_{Y_i})$ is klt outside 
$\Supp \rddown{B_{Y_i}}$. Since $K_X+B\sim_\Q 0$ over $U$ and since $\rddown{B}$ is vertical$/Z$, 
there are vertical$/Z$ $\Q$-divisors $M,N\ge 0$ such that 
$$
K_X+B\sim_\Q M/Z ~~~~~~\mbox{and} ~~~~~~-\rddown{B}\sim_\Q N/Z
$$ 
and
$$
K_X+B-t_i\rddown{B}\sim_\Q M+t_iN/Z
$$ 
So, any prime divisor contracted 
by $X\bir Y_i$ is a component of $M+N$ hence after replacing the sequence with a subsequence 
we can assume that the maps $X\bir Y_i$ contract the same divisors, i.e. $Y_i$ are 
isomorphic in codimension one.

\emph{Step 3.} Assume that $Y_i=Y_{i+1}$ for $i\gg 0$, and let $Y$ be this common model.
Since $K_Y+B_Y-t_i\rddown{B_{Y}}$ is nef$/Z$ for each $i\gg 0$, $K_Y+B_Y$ is also nef$/Z$. 
For any prime divisor $D$ on $X$ and each $i\gg 0$, we have 
$$
a(D,X,B)\le a(D,X,B-t_i\rddown{B})\le a(D,Y,B_Y-t_i\rddown{B_Y}) 
$$
which implies that 
$$
a(D,X,B)\le \lim_{i\to \infty} a(D,Y,B_Y-t_i\rddown{B_Y})=a(D,Y,B_Y)
$$
Therefore, $(Y/Z,B_Y)$ is a weak lc model of $(X/Z,B)$, and  
by Corollary \ref{c-wlc-to-mmodel}, we can construct a log minimal model of $(X/Z,B)$ 
as required.

\emph{Step 4.}
Assume that $\rddown{B_{Y_1}}=0$. Then,  $K_{Y_1}+B_{Y_1}-t_i\rddown{B_{Y_1}}$ is nef$/Z$ 
for each $i$, and since $Y_1\bir Y_i$ is an isomorphism in codimension one,  
we can replace each $Y_i$ with $Y_1$ and then apply Step 3.

 From now on we assume that $\rddown{B_{Y_1}}\neq 0$. 
Then, 
since $(Y_1/Z,B_{Y_1}-t_1\rddown{B_{Y_1}})$ is klt and $K_{Y_1}+B_{Y_1}-t_1\rddown{B_{Y_1}}$ is 
nef$/Z$, $K_{Y_1}+B_{Y_1}-t_1\rddown{B_{Y_1}}$ is semi-ample$/Z$ by Theorem \ref{t-main-klt}.
Thus,
$$
K_{Y_1}+B_{Y_1}-t_1\rddown{B_{Y_1}}\sim_\Q C_{Y_1}/Z
$$
for some $C_{Y_1}\ge 0$ such that 
$K_{Y_1}+B_{Y_1}+C_{Y_1}$ is lc, in particular, $\Supp C_{Y_1}$ does not contain any lc centre 
of $(Y_1/Z,B_{Y_1})$. For each 
rational number $t$ we have 
$$
K_{Y_1}+B_{Y_1}+tC_{Y_1}\sim_\Q (1+t)(K_{Y_1}+B_{Y_1}-\frac{tt_1}{1+t}\rddown{B_{Y_1}})/Z
$$
In particular, if $t_i'=\frac{t_i}{t_1-t_i}$ and $i\neq 1$, then 
$$
K_{Y_1}+B_{Y_1}+t_i'C_{Y_1}\sim_\Q (1+t_i')(K_{Y_1}+B_{Y_1}-t_i\rddown{B_{Y_1}})/Z
$$
We may assume that $2t_i\le t_1$ for each $i>1$ which implies that $t_i'\in [0,1]$. So, $(Y_i/Z,B_{Y_i}+t_i'{C_{Y_i}})$ is a weak lc model of $(Y_1/Z,B_{Y_1}+t_i'{C_{Y_1}})$ for each $i>1$.
Note that $(Y_1/Z,B_{Y_1}+t_i'{C_{Y_1}})$ is lc because $t_i'\le 1$.

\emph{Step 5.} We will modify the situation so that 
if  $\eta$ is the 
generic point of any lc centre of $({Y_2}/Z,B_{Y_2})$, then $\eta$ is mapped into $U$. 
Put $P_{Y_2}=t_1\rddown{B_{Y_2}}$. Then, 
$$
K_{Y_2}+B_{Y_2}\sim_\Q P_{Y_2}+C_{Y_2}/Z
$$
By Lemma \ref{l-LMMP-scaling}, we can run an LMMP$/Z$ on $K_{Y_2}+B_{Y_2}$ with 
scaling of $t_2'C_{Y_2}$ so that if $\lambda_k$ are the numbers appearing in the LMMP, then $\lambda=\lim_{k\to \infty} \lambda_k=0$. Since $Y_2\bir Y_i$ is an isomorphism in codimension one,   
$K_{Y_i}+B_{Y_i}+t_i'C_{Y_i}$ is nef$/Z$, and $\lim_{i\to \infty} t_i'=0$, we deduce that 
$K_{Y_2}+B_{Y_2}$ is (numerically) a limit of movable$/Z$ $\R$-divisors. 
So, the LMMP does not contract any divisors.
Now, by replacing the $t_i$ with a subsequence,
we can assume that for each $i>2$ there is $k$ such that 
$\lambda_k\ge t_i'\ge \lambda_{k+1}$. If $T$ is the variety corresponding to $\lambda_{k+1}$, 
then $K_T+B_T+t_i'C_T$ is nef$/Z$. By replacing $Y_i$ with $T$, we could assume that 
$Y_i$ occurs in some step of the LMMP. Thus, we can assume that every $Y_i$ 
occurs in some step of the LMMP if $i>1$.

After finitely many steps, the LMMP does not contract any lc centres. So, 
perhaps after replacing $({Y_2}/Z,B_{Y_2})$ with some $({Y_j}/Z,B_{Y_j})$, with $j\gg 0$,
we may assume that the LMMP does not contract any lc centre of $({Y_2}/Z,B_{Y_2})$. 
Assume that $\eta$ is the 
generic point of some lc centre of $({Y_2}/Z,B_{Y_2})$ which is mapped outside $U$. 
If $E$ is a prime divisor over $\eta$ such that 
$a(E,{Y_2},B_{Y_2})=0$, then $a(E,X,B)>0$ and for any $i\ge 2$ we have 
 \begin{equation*}
\begin{split}
a(E,{Y_2},B_{Y_2}-t_i\rddown{B_{Y_2}})= & a(E,{Y_i},B_{Y_i}-t_i\rddown{B_{Y_i}})\\
\ge & a(E,X,B-t_i\rddown{B})\\
\ge & a(E,X,B)
\end{split}
\end{equation*}
where the equality follows from the fact that $Y_2\bir Y_i$ is an isomorphism near $\eta$.
Then, 
$$
0=a(E,{Y_2},B_{Y_2})=\lim_{i\to \infty} a(E,{Y_2},B_{Y_2}-t_i\rddown{B_{Y_2}})\ge a(E,X,B)>0
$$
which is a contradiction. So, from now on we can assume that 
if $\eta$ is the generic point of any lc centre of $({Y_2}/Z,B_{Y_2})$, then $\eta$ is mapped into $U$.
Moreover, by construction, $K_{Y_2}+B_{Y_2}\sim_\Q 0$ over $U$.

\emph{Step 6.} 
By Lemma \ref{l-acc-s-termination}, we can run an 
LMMP$/Z$ on $K_{Y_2}+B_{Y_2}$ with scaling of $t_2'C_{Y_2}$ 
which terminates on a model $Y$ on which $K_Y+B_Y+\delta C_Y$ 
is nef$/Z$ for any sufficiently small $\delta\ge 0$. 
Since $Y_2\bir Y_i$ is an isomorphism in codimension one,   
$K_{Y_i}+B_{Y_i}+t_i'C_{Y_i}$ is nef$/Z$, and $\lim_{i\to \infty} t_i'=0$, we deduce that 
$K_{Y_2}+B_{Y_2}$ is (numerically) a limit of movable$/Z$ $\R$-divisors. 
So, the LMMP does not contract any divisors, i.e. $Y\bir Y_i$ is an isomorphism in codimension one.

For any $i\gg 0$, 
$K_Y+B_{Y}+t_i'{C_{Y}}$ is nef$/Z$ which in turn implies that 
$K_Y+B_{Y}-t_i\rddown{B_Y}$ is also nef$/Z$.
Since $Y\bir Y_i$ is an isomorphism in codimension one and 
since  $(Y_i/Z,B_{Y_i}-t_i\rddown{B_{Y_i}})$ is a log minimal model of 
 $(X/Z,B-t_i\rddown{B})$,  $(Y/Z,B_Y-t_i\rddown{B_Y})$ is a log minimal model of 
$(X/Z,B-t_i\rddown{B})$, for every $i\gg 0$.
By Step 3, we are done.
\end{proof}

\begin{prop}\label{p-acc-main-vertical-sa}
Theorem \ref{t-acc-main-general-modified} in dimension $d-1$ implies Theorem \ref{t-acc-main-general-modified} 
 in dimension $d$ in the vertical case, i.e. when every lc centre of $(X/Z,B)$  is vertical$/Z$.
\end{prop}
\begin{proof}
By Proposition \ref{p-acc-main-vertical}, statements (1) and (3) of Theorem \ref{t-acc-main-general-modified} 
hold for $(X/Z,B)$, in particular,  $(X/Z,B)$ has  
a log minimal model $(Y/Z,B_Y)$. By assumptions, $K_X+B\sim_\Q 0$ over some non-empty
open subset $U\subseteq Z$ and if $\eta$ 
is the generic point of any lc centre of $(X/Z,B)$ then $f(\eta)\in U$. This implies that 
$K_Y+B_Y\sim_\Q 0$ over $U$ and that if $\eta$ 
is the generic point of any lc centre of $(Y/Z,B_Y)$ then $\eta$ is mapped into $U$.
By replacing $(X/Z,B)$ with $(Y/Z,B_Y)$ we may assume that $(X/Z,B)$ 
is $\Q$-factorial dlt and that $K_X+B$  is nef$/Z$. It remains to prove that 
$K_X+B$ is semi-ample$/Z$. 

Run an LMMP$/Z$ on $K_X+B-\epsilon \rddown{B}$ with scaling of some ample$/Z$ 
divisor, for some sufficiently 
small rational number $\epsilon>0$. 
By [\ref{B-II}, Proposition 3.2], $K_X+B$ is numerically trivial on each 
step of this LMMP. Since $\rddown{B}$ is vertical$/Z$, $(K_X+B-\epsilon \rddown{B})|_F\sim_\Q 0$ 
where $F$ is the generic fibre of $f$. So, 
by Theorem \ref{t-main-klt}, the LMMP terminates 
on a model $X'$ on which $K_{X'}+B_{X'}$ and $K_{X'}+B_{X'}-\epsilon \rddown{B_{X'}}$
are both nef$/Z$. Another application of Theorem \ref{t-main-klt} shows that 
$K_{X'}+B_{X'}-\delta \rddown{B_{X'}}$ is semi-ample$/Z$ for any $\delta\in (0,\epsilon]$.
The pair $(X'/Z,B_{X'})$ is $\Q$-factorial and lc and  $(X'/Z,B_{X'}-\epsilon \rddown{B_{X'}})$ 
is klt. If $\rddown{B_{X'}}=0$ we are done so we can assume that $\rddown{B_{X'}}\neq 0$.
Let $(Y/Z,B_Y)$ be a $\Q$-factorial dlt blowup of $(X'/Z,B_{X'})$. By construction, 
$K_Y+B_Y\sim_\Q 0$ over $U$ and if $\eta$ is the generic point of any lc 
centre of $(Y/Z,B_Y)$ then $\eta$ is mapped into $U$. 

Let $S$ be a component of $T:=\rddown{B_{Y}}$ and put
$K_S+B_S:=(K_Y+B_Y)|_S$. If $Q$ is the normalisation of the image of $S$ in $Z$ and if 
$V\subseteq Q$ is the inverse image of $U$, then $K_S+B_S\sim_\Q 0$ over $V$ 
and if $\eta$ is the generic point of any lc 
centre of $(S/Q,B_S)$ then $\eta$ is mapped into $V$.
So, by induction, $K_S+B_S$ is semi-ample$/Q$ hence semi-ample$/Z$.  On the other hand, if 
$P:= e^*\rddown{B_{X'}}$ and if $\delta>0$ is any sufficiently 
small rational number, where $e$ is the morphism $Y\to X'$, then 
$$
K_Y+B_Y-\delta P=e^*(K_{X'}+B_{X'}-\delta \rddown{B_{X'}})
$$ 
is semi-ample$/Z$. Now, by Theorem \ref{conj-s-ampleness}, 
$K_Y+B_Y$ is semi-ample$/Z$ hence $K_X+B$ is also semi-ample$/Z$.
\end{proof}

In the proof of \ref{p-acc-main-vertical-sa}, it is also possible to compactify 
$X$ and $Z$ first before running the LMMP on $K_X+B-\epsilon \rddown{B}$; in this 
way we would need Theorem \ref{conj-s-ampleness} only 
when $X,Z$ are projective.

%%%%%%%%%%%%%%%%%%%%%%%%%%%%%%%%%%%%%

\subsection*{The horizontal case} In this subsection, we deal with the horizontal case of 
Theorem \ref{t-acc-main-general-modified}. 

\begin{lem}\label{l-acc-horizontal-trivial-1}
Assume Conjecture \ref{ACC} and Theorem \ref{t-acc-main-general-modified} in dimension $d-1$.
Let $(X/Z,B)$ be of dimension $d$ as in Theorem \ref{t-acc-main-general-modified}.
Moreover, assume that there is a contraction $g\colon X\to T/Z$ such that 

$(1)$ $K_X+B\sim_\Q 0/T$,

$(2)$ some lc centre of $(X/Z,B)$ is horizontal over $T$.\\
Then, $(X/Z,B)$ has a good log minimal model.
\end{lem}
\begin{proof}
By replacing $(X/Z,B)$ with a $\Q$-factorial dlt blowup, 
we can assume that $(X/Z,B)$ is $\Q$-factorial dlt and that there is a 
component $S$ of $\rddown{B}$ which is horizontal$/T$. Run an LMMP$/Z$ on 
$K_X+B$ with scaling of some ample$/Z$ divisor. Since termination and semi-ampleness$/Z$ are 
local on $Z$, 
we can assume that $Z$ is affine, say $\Spec R$. 
By Theorem \ref{t-mmodel-term-scaling}, 
the LMMP terminates with a 
good log minimal model if we prove that $(X/Z,B)$ has a good log minimal model. 

By assumptions, $K_X+B\sim_\Q 0$ over some non-empty open subset $U\subseteq Z$ and if $\eta$ 
is the generic point of any lc centre of $(X/Z,B)$ then $f(\eta)\in U$.
By adjunction define $K_S+B_S:=(K_X+B)|_S$.
Then, $K_S+B_S\sim_\Q 0$ over $U$ and if $\eta$ 
is the generic point of any lc centre of $(S/Z,B_S)$ then $f(\eta)\in U$. 
Moreover, $K_S+B_S\sim_\Q 0/T$.

Since we are assuming Conjecture \ref{ACC} and Theorem \ref{t-acc-main-general-modified} in dimension $d-1$, $(S/Z,B_S)$ has a good log minimal model. 
Therefore, if $I(K_S+B_S)$ is Cartier for some $I\in\N$, 
then  $R(S/Z,I(K_S+B_S))$ is a finitely generated $R$-algebra (cf. [\ref{B-dam}]).
We can choose $I$ such that $I(K_X+B)$ is Cartier and such that 
$I(K_X+B)\sim g^*L$ for some Cartier divisor $L$ on $T$. So, 
$I(K_S+B_S)\sim e^*L$ where $e\colon S\to T$ is the induced morphism.
Now, by Lemma \ref{t-fg-pullbacks}, $R(T/Z,L)$ is a finitely 
generated $R$-algebra which in turn implies that $R(X/Z,I(K_X+B))$
is a finitely generated $R$-algebra since $X\to T$ is a contraction. 
Therefore, $R(X/Z,K_X+B)$
is a finitely generated $R$-algebra and according to Theorem \ref{t-main-under-fg}, $(X/Z,B)$ has a good log minimal model.
\end{proof}

\begin{prop}\label{t-acc-main-horizontal}
Assume Theorem \ref{t-acc-main-general-modified}  in dimension $d-1$.
Then, Theorem \ref{t-acc-main-general-modified} holds in dimension $d$ in the horizontal case, i.e. 
when some lc centre of $(X/Z,B)$ is horizontal over $Z$.
\end{prop}
\begin{proof}
We can assume that $(X/Z,B)$ is $\Q$-factorial dlt and that $f\colon X\to Z$ is a contraction. 
By assumptions, some component of $\rddown{B}$ is 
horizontal$/Z$. Moreover, $K_{X}+B\sim_\Q 0$ over some non-empty open subset $U\subseteq Z$ and 
if $\eta$ is the generic point of any lc centre of $(X/Z,B)$ then $\eta$ is mapped into $U$. 
Run an LMMP$/Z$ on $K_X+B$ with scaling of some ample$/Z$ divisor. 
By Theorem \ref{t-mmodel-term-scaling}, the LMMP terminates with a good log minimal model 
if we show that $(X/Z,B)$ has a 
good log minimal model.

If $\epsilon>0$ is a sufficiently small rational number, then 
$(X/Z,B-\epsilon \rddown{B})$ is klt and $K_X+B-\epsilon \rddown{B}$ is not pseudo-effective$/Z$. 
Thus, by [\ref{BCHM}] there is a 
Mori fibre space $(Y/Z,B_Y-\epsilon \rddown{B_Y})$ for $(X/Z,B-\epsilon \rddown{B})$ 
obtained by running an LMMP$/Z$ on $K_X+B-\epsilon \rddown{B}$. 
Let $g\colon Y\to T/Z$ be the $K_Y+B_Y-\epsilon \rddown{B_Y}$-negative extremal 
contraction which defines the Mori fibre space structure, and let $R$ be the corresponding 
extremal ray. By construction, 
$K_Y+B_Y\sim_\Q 0$ over $U$ hence $(K_Y+B_Y)\cdot R=0$ which in turn implies that 
 $K_Y+B_Y\sim_\Q 0/T$. Since 
$K_Y+B_Y-\epsilon \rddown{B_Y}$ is numerically negative$/T$,  $\rddown{B_Y}\neq 0$.
As we are assuming Conjecture \ref{ACC} in dimension $d$, we can assume that $(Y/Z,B_Y)$ is lc.

Let $(W/Z,B_W)$ be a log smooth model of $(X/Z,B)$ of type (2) as in Definition \ref{d-log-smooth-model} 
such that $W$ dominates $Y$, say by a morphism $h\colon W\to Y$. In particular, we may assume that 
if $\eta$ is the generic point 
of any lc centre of $(W/Z,B_W)$, then $\eta$ is mapped into $U$. 
 We can write 
$$
K_W+B_W=h^*(K_Y+B_Y)+G
$$ 
where $G$ is exceptional$/Y$. Run the LMMP$/Y$ on $K_W+B_W$ with scaling of some ample$/Y$ divisor. 
By Theorem \ref{t-exc-LMMP}, we reach a model $X'$ on which $G_{X'}\le 0$. Moreover, 
$$
K_{X'}+B_{X'}-G_{X'}\sim_\Q 0/T
$$ as it is the pullback of $K_Y+B_Y$. 

Let $\pi$ be the given morphism $W\to X$. By definition of log smooth models, we can write 
$$
K_W+B_W=\pi^*(K_X+B)+E
$$ 
where $E$ is effective and exceptional$/X$. On the other hand,  
since $K_X+B\sim_\Q 0$ and $K_Y+B_Y\sim_\Q 0$ over $U$, 
we have $\pi^*(K_X+B)=h^*(K_Y+B_Y)$ over $U$ which means that $E=G$ over $U$. 
Therefore, every component of $G$ with negative coefficient is mapped into $Z\setminus U$ 
hence $G_{X'}$ is also mapped into $Z\setminus U$. In particular, this means that 
 $K_{X'}+B_{X'}\sim_\Q 0$ 
over $U$. Moreover, if $\eta$ 
is the generic point of any lc centre of $(X'/Z,B_{X'})$ then $\eta$ is mapped into $U$ 
because the same holds for $(W/Z,B_W)$. 
Let $U_T\subseteq T$ be the inverse image of $U$. Then, 
$K_{X'}+B_{X'}\sim_\Q 0$ 
over $U_T$ and if $\eta$ 
is the generic point of any lc centre of $(X'/Z,B_{X'})$ then $\eta$ is mapped into $U_T$. 
 
By construction, $\rddown{B_Y}$ is ample$/T$ hence in particular some of its components are horizontal$/T$. 
This implies that some component of $\rddown{B_{X'}}$ is horizontal$/T$ because $\rddown{B_Y}$ is the pushdown of 
$\rddown{B_{X'}}$.
By Lemma \ref{l-acc-horizontal-trivial-2} below,  we can run an LMMP$/T$ on $K_{X'}+B_{X'}$ 
which ends up with a model $X''$ 
on which $K_{X''}+B_{X''}$ is semi-ample$/T$ (more precisely, 
we should take $X,Y,Z,B$ in the lemma to be our $X',Y,T,B_{X'}$ respectively). 
Over $U_T$, $X'\bir X''$ is an isomorphism.
So, if $X''\to T''/T$ is the contraction associated 
to $K_{X''}+B_{X''}$, then $T''\to T$ is birational and some component of $\rddown{B_{X''}}$ is horizontal$/T''$. 

By construction, $K_{X''}+B_{X''}\sim_\Q 0/T''$, and $K_{X''}+B_{X''}\sim_\Q 0$ over $U$. Moreover, 
if $\eta$ 
is the generic point of any lc centre of $(X''/Z,B_{X''})$ then $\eta$ is mapped into $U$. 
 Now we can apply Lemma \ref{l-acc-horizontal-trivial-1} (by taking $X,Z,T,B$ in the lemma 
 to be our $X'',Z,T'',B_{X''}$ respectively) to get a good log minimal model of $({X''}/Z,B_{X''})$ 
 which would give a good log minimal model of $(W/Z,B_W)$ hence of $(X/Z,B)$, by Remark \ref{rem-log-smooth-model}.
\end{proof}

\begin{lem}\label{l-acc-horizontal-trivial-2}
Assume Conjecture \ref{ACC} in dimension $d$ and Theorem \ref{t-acc-main-general-modified} in dimension $d-1$.
Let $(X/Z,B)$ be of dimension $d$ as in Theorem \ref{t-acc-main-general-modified}.
Moreover, assume that

$\bullet$ $(X/Z,B)$ is $\Q$-factorial dlt,

$\bullet$ there exist contractions $e\colon X\to Y$ and $g\colon Y\to Z$,

$\bullet$ $e$ is birational and $Y$ is $\Q$-factorial, 

$\bullet$ $K_Y+B_Y$ is lc, $D:=K_Y+B_Y-\epsilon \rddown{B_Y}$ is klt, and $g$ is a $D$-negative 
extremal contraction for some $\epsilon>0$ where $B_Y:=e_*B$.

Then, any LMMP$/Z$ on $K_X+B$ with scaling of an ample$/Z$ divisor  
terminates with a good log minimal model of $(X/Z,B)$.
\end{lem}
\begin{proof}
By Theorem \ref{t-mmodel-term-scaling}, it is enough to prove that $(X/Z,B)$ has a 
good log minimal model. By assumptions, $K_X+B\sim_\Q 0$ 
over some non-empty open subset $U\subseteq Z$ and if $\eta$ 
is the generic point of any lc centre of $(X/Z,B)$ then $f(\eta)\in U$. 
Let $U_Y\subseteq Y$ be the inverse image of $U$. Then, 
$K_X+B\sim_\Q 0$ 
over  $U_Y$ and if $\eta$ 
is the generic point of any lc centre of $(X/Z,B)$ then $e(\eta)\in U_Y$.
Since $e$ is birational, by Proposition \ref {p-acc-main-vertical-sa} we can assume
 that $K_X+B$ is actually 
semi-ample$/Y$. By replacing $X$ with the lc model of $(X/Y,B)$ we can 
assume that $K_X+B$ is ample$/Y$ (we may loose the $\Q$-factorial dlt property of $(X/Z,B)$
but we will recover it later).

Since $g$ is an extremal contraction and since $K_Y+B_Y\sim_\Q 0$ over $U$, 
$K_Y+B_Y\sim_\Q 0/Z$. Moreover, 
since  $K_X+B$ is ample$/Y$, there is an $e$-exceptional $\Q$-divisor $A\ge 0$ 
such that 
$$
K_X+B+A=e^*(K_Y+B_Y)
$$ 
and $K_X+B+A\sim_\Q 0/Z$, and since $K_Y+B_Y$ is lc, $K_X+B+A$ is lc too. Also, since $-A$ is 
ample$/Y$, $\Supp A$ contains all the prime exceptional$/Y$ divisors on $X$. 
Now we can add a small multiple of $A$ to $B$ and assume that $\Supp B$ 
contains  all the prime exceptional$/Y$ divisors on $X$. Note that 
$A\sim_\Q 0$ over $U_Y$ hence $A=0$ over $U_Y$, so we still have $K_X+B\sim_\Q 0$ 
over $U$ after adding a small multiple of $A$ to $B$. Now, 
$$
\Supp e^*\rddown{B_Y}\subseteq (\Supp \rddown{B}\cup \Supp A) \subseteq \Supp B
$$ 
Replacing $(X/Z,B)$ with a 
$\Q$-factorial dlt blowup, we can again assume that $(X/Z,B)$ is $\Q$-factorial dlt; 
the property $\Supp e^*\rddown{B_Y}\subseteq \Supp B$ is preserved.

By assumptions, $\rddown{B_Y}$ is 
ample$/Z$ hence $e^*\rddown{B_Y}$ 
is semi-ample$/Z$. Thus, for a small rational number $\tau >0$ we can write
$$
K_X+B=K_X+B-\tau e^*\rddown{B_Y}+\tau e^*\rddown{B_Y}\sim_\Q K_X+\Delta/Z
$$ 
where $\Delta$ is some rational boundary such that 
$(X/Z,\Delta)$ is klt:  since 
$(Y/Z,B_Y-\epsilon \rddown{B_Y})$ is klt, $\rddown{B_Y}$ contains all 
the lc centres of $(Y/Z,B_Y)$, in particular, the image of all the lc centres 
of $(X/Z,B)$; so, $\Supp e^*\rddown{B_Y}$ contains all the components of 
$\rddown{B}$ hence $(X/Z,B-\tau e^*\rddown{B_Y})$ is klt and we can indeed 
find $\Delta$ with the required properties.
By Theorem \ref{t-main-klt}, we can 
run an LMMP$/Z$ on  $K_X+\Delta$ which ends up with a good log minimal model 
of $(X/Z,\Delta)$ hence a good log minimal model of $(X/Z,B)$.
\end{proof}

\begin{proof}(of Theorem \ref{t-acc-main-general-modified})
This follows from Propositions \ref{p-acc-main-vertical-sa} and \ref{t-acc-main-horizontal}.
\end{proof}

%%%%%%%%%%%%%%%%%%%%%%%%%%%%%%%%%%%%%

%%%%%%%%%%%%%%%%%%%%%%%%%%%%%%%%%%%%%
%%%%%%%%%%%%%%%%%%%%%%%%%%%%%%%%%%%%%
%%%%%%%%%%%%%%%%%%%%%%%%%%%%%%%%%%%%%

%%%%%%%%%%%%%%%%%%%%%%%%%%%%%%%%%%%%%

\vspace{2cm}

\flushleft{DPMMS}, Centre for Mathematical Sciences,\\
Cambridge University,\\
Wilberforce Road,\\
Cambridge, CB3 0WB,\\
UK\\
email: c.birkar@dpmms.cam.ac.uk

\end{document}